\documentclass[a4paper,10pt,reqno]{amsart}
\usepackage[utf8x]{inputenc}

\usepackage{amsmath,amssymb,amsthm,amsxtra,mathrsfs,bm,bbm}
\usepackage{graphicx}
\usepackage[authoryear]{natbib}
\usepackage{url}
\usepackage[colorlinks,citecolor=blue]{hyperref}

\newcommand{\ee}{\mathrm{e}}
\newcommand{\ii}{\mathrm{i}}

\newcommand{\Z}{\mathbb{Z}}

\newcommand{\R}{\mathbb{R}}

\newcommand{\uu}{\boldsymbol{u}}
\newcommand{\vv}{\boldsymbol{v}}
\newcommand{\ww}{\boldsymbol{w}}
\newcommand{\zz}{\boldsymbol{z}}
\newcommand{\BB}{\boldsymbol{B}}
\newcommand{\ff}{\boldsymbol{f}}
\newcommand{\hh}{\boldsymbol{h}}

\newcommand{\Sch}{\mathscr{S}}
\newcommand{\Ss}{\mathcal{S}}

\newcommand{\weakstarto}{\overset{*}{\rightharpoonup}}

\newcommand{\Loc}{\mathord{\bigtriangleup}}
\newcommand{\loc}{\mathord{\dot{\bigtriangleup}}}

\newcommand{\pd}{\partial}
\newcommand{\rd}{\mathrm{d}}

\newcommand{\Grad}{\nabla}
\newcommand{\Div}{\nabla \cdot}

\newcommand{\Laplace}{\Delta}

\DeclareMathOperator*{\supp}{supp}

\newcommand{\abs}[1]{\left| #1 \right|}
\newcommand{\norm}[1]{\| #1 \|}
\newcommand{\bignorm}[1]{\left\| #1 \right\|}

\newcommand{\inner}[2]{\langle #1 , #2 \rangle}

\theoremstyle{plain}
\newtheorem{thm}{Theorem}[section]
\newtheorem{prop}[thm]{Proposition}
\newtheorem{lemma}[thm]{Lemma}

\numberwithin{equation}{section}

\title[Local existence for non-resistive MHD in Besov spaces]{Local existence for the non-resistive MHD equations in Besov spaces}

\author[J.-Y.\ Chemin]{Jean-Yves Chemin}
\thanks{DSMcC was a member of the Warwick ``MASDOC'' doctoral training centre, funded by EPSRC grant EP/HO23364/1. JCR was partially supported by an EPSRC Leadership Fellowship EP/G007470/1. JLR is partially supported by the European Research Council, grant no.\ 616797.}
\address{J.-Y.\ Chemin \\
Laboratoire Jacques Louis Lions - UMR 7598 \\
Universit\'{e} Pierre et Marie Curie-Paris 6 \\
Bo\^{i}te courrier 187 \\
4 place Jussieu \\
75252 Paris cedex 05 \\
France}
\email{chemin@ann.jussieu.fr}

\author[D.\ S.\ McCormick]{David S.\ McCormick}
\address{D.\ S.\ McCormick \\
Mathematics Institute \\
University of Warwick \\
Coventry, CV4 7AL \\
United Kingdom}
\email{d.s.mccormick@warwick.ac.uk}

\author[J.\ C.\ Robinson]{James C.\ Robinson}
\address{J.\ C.\ Robinson \\
Mathematics Institute \\
University of Warwick \\
Coventry, CV4 7AL \\
United Kingdom}
\email{j.c.robinson@warwick.ac.uk}

\author[J.\ L.\ Rodrigo]{Jose L.\ Rodrigo}
\address{J.\ L.\ Rodrigo \\
Mathematics Institute \\
University of Warwick \\
Coventry, CV4 7AL \\
United Kingdom}
\email{j.rodrigo@warwick.ac.uk} 

\date{March 5, 2015}



\makeatletter
\hypersetup{
	pdftitle = {\@title},
	pdfauthor = {\shortauthors},
	pdfsubject = {2010 MSC: \@subjclass},
	pdfkeywords = {\@keywords},
}
\makeatother

\begin{document}

\begin{abstract}
In this paper we prove the existence of solutions to the viscous, non-resistive magnetohydrodynamics (MHD) equations on the whole of $\R^{n}$, $n=2,3$, for divergence-free initial data in certain Besov spaces, namely $\uu_{0} \in B^{n/2-1}_{2,1}$ and $\BB_{0} \in B^{n/2}_{2,1}$. The a priori estimates include the term $\int_{0}^{t} \norm{\uu(s)}_{H^{n/2}}^{2} \, \rd s$ on the right-hand side, which thus requires an auxiliary bound in $H^{n/2-1}$. In 2D, this is simply achieved using the standard energy inequality; but in 3D an auxiliary estimate in $H^{1/2}$ is required, which we prove using the splitting method of Calder\'on (\textit{Trans.\ Amer.\ Math.\ Soc.}\ \textbf{318}(1), 179--200, 1990). By contrast, we prove that such solutions are unique in 3D, but the proof of uniqueness in 2D is more difficult and remains open.
\end{abstract}

\maketitle

\section{Introduction}

In this paper we prove local-in-time existence of weak solutions to the non-resistive magnetohydrodynamics (MHD) equations:
\begin{subequations}
\label{eqn:MHD-NonRes}
\begin{align}
\frac{\pd \uu}{\pd t} + (\uu \cdot \Grad) \uu - \nu \Laplace \uu + \Grad p_{*} &= (\BB \cdot \Grad) \BB, \label{eqn:MHD-NonRes-u} \\
\frac{\pd \BB}{\pd t} + (\uu \cdot \Grad) \BB &= (\BB \cdot \Grad) \uu, \label{eqn:MHD-NonRes-B} \\
\Div \uu = \Div \BB &= 0, \label{eqn:MHD-NonRes-Div}
\end{align}
\end{subequations}
on the whole of $\R^{n}$ for $n = 2, 3$, with divergence-free initial data in Besov spaces as follows:
\[
\uu_{0} \in B^{n/2-1}_{2,1}(\R^{n}) \qquad \text{and} \qquad \BB_{0} \in B^{n/2}_{2,1}(\R^{n}).
\]
In particular, we prove the following theorem.

\begin{thm}
\label{thm:MHDLocalExistence-Besov}
Let $n=2,3$. For $\uu_{0} \in B^{n/2-1}_{2,1}(\R^{n})$ and $\BB_{0} \in B^{n/2}_{2,1}(\R^{n})$ with $\Div \uu_{0} = \Div \BB_{0} = 0$, there exists a time $T_{*} = T_{*}(\nu, \uu_{0}, \norm{\BB_{0}}_{B^{n/2}_{2,1}}) > 0$ such that the equations \eqref{eqn:MHD-NonRes} have at least one weak solution $(\uu, \BB)$, with 
\begin{align*}
\uu &\in L^{\infty}([0, T_{*}]; B^{n/2-1}_{2,1}(\R^{n})) \cap L^{1}(0, T_{*}; B^{n/2+1}_{2,1}(\R^{n})), \\
\BB &\in L^{\infty}([0, T_{*}]; B^{n/2}_{2,1}(\R^{n})).
\end{align*}
\end{thm}

This result is the natural generalisation of the main result of \cite{art:Commutators}, in which local-in-time existence of strong solutions to \eqref{eqn:MHD-NonRes} was proved on the whole of $\R^{n}$ with $n=2, 3$, with divergence-free initial data $\uu_{0}, \BB_{0} \in H^{s}(\R^{n})$, for $s > n/2$. This depended upon a commutator estimate, a partial generalisation of that of \cite{art:KatoPonce1988}, which does not hold for $s = n/2$.

In this paper, we work instead in the space $B^{n/2}_{2,1}$, which is the natural replacement for the space $H^{n/2}$: it is the largest Besov space which still embeds in $L^{\infty}$ (unlike $H^{n/2}$). Thanks to the properties of the heat equation in Besov spaces, we require one fewer derivative for the initial data $\uu_{0}$, requiring only that $\uu_{0} \in B^{n/2-1}_{2,1}(\R^{n})$; but with no diffusion term in the $\BB$ equation we still require $\BB_{0} \in B^{n/2}_{2,1}(\R^{n})$.

This paper, like \cite{art:Commutators}, builds on a number of previous results for the non-resistive MHD equations, including \cite{art:JiuNiu2006}, \cite{art:FanOzawa2009} and \cite{art:ZhouFan2011}. Moreover, for the fully ideal MHD equations (with no diffusion in either equation), \cite{art:MiaoYuan2006} proved existence and uniqueness of solutions to fully ideal MHD in the Besov space $B^{1+n/p}_{p,1}(\R^{n})$.

Nonetheless, the results for the non-resistive equations are still much weaker than those for the fully diffusive MHD equations, in which the term $-\eta \Laplace \BB$ appears in \eqref{eqn:MHD-NonRes-B}: in 2D one has global existence and uniqueness of weak solutions, and in 3D one has local existence of weak solutions, much like the Navier--Stokes equations; these results go back to \cite{art:DuvautLions1972} and \cite{art:SermangeTemam1983}. A detailed discussion of previous work on the subject can be found in the introduction to \cite{art:Commutators}.

The rest of the paper is structured as follows:

\begin{itemize}
\item In Section~\ref{sec:Besov}, we recall some of the theory of Besov spaces used throughout the paper.
\item In Section~\ref{sec:APriori2D3D}, we prove two of the key a priori estimates necessary in the proof of Theorem~\ref{thm:MHDLocalExistence-Besov}: these two estimates apply equally in both 2D and 3D.
\item In Section~\ref{sec:Bounds}, we prove additional estimates on the term $\int_{0}^{T} \norm{\uu(t)}_{H^{n/2}}^{2} \, \rd t$, which appears on the right-hand side of the estimate for the $\uu$ equation proved in Section~\ref{sec:APriori2D3D}, in order to close up the a priori estimates. Different arguments are required in 2D and 3D.
\begin{itemize}
\item In 2D, this is easily taken care of using the energy inequality (see Section~\ref{sec:Besov2D}).
\item In 3D, this needs a careful argument, based on the splitting method of \cite{art:Calderon1990}, to yield an $H^{1/2}$ estimate for the Navier--Stokes equations (see Section~\ref{sec:Besov3D}).
\end{itemize}
\item In Section~\ref{sec:Outline}, with the necessary estimates completed, the rest of the proof of Theorem~\ref{thm:MHDLocalExistence-Besov} is outlined.
\item In Section~\ref{sec:Uniqueness-Besov} we prove that, in 3D, the solution whose existence is asserted by Theorem~\ref{thm:MHDLocalExistence-Besov} is unique.
\end{itemize}

Surprisingly, the proof of uniqueness in 2D is more difficult and remains open. Furthermore, note that we require the initial data to have finite energy, taking $\uu_{0}$ and $\BB_{0}$ in inhomogeneous Besov spaces rather than their homogeneous counterparts. For further discussion on both these issues, see the conclusion (Section~\ref{sec:Conclusion}).

\section{Besov Spaces}
\label{sec:Besov}

Here we recall some of the standard theory of Besov spaces which we will use throughout the paper; we use, as far as possible, the same notation as \cite*{book:BCD2011}, and refer the reader to Chapter 2 therein for proofs and many more details that we must omit.

\subsection{Definitions}
\label{sec:BesovDefns}

For the purposes of this section, given a function $\phi$ and $j \in \Z$ we denote by $\phi_{j}$ the dilation
\[
\phi_{j}(\xi) = \phi(2^{-j}\xi).
\]
Let $\mathcal{C}$ be the annulus $\{ \xi \in \R^{n} : 3/4 \leq |\xi| \leq 8/3 \}$. There exist radial functions $\chi \in C^{\infty}_{c}(B(0, 4/3))$ and $\varphi \in C^{\infty}_{c}(\mathcal{C})$ both taking values in $[0,1]$ such that
\begin{subequations}
\label{eqn:DyadicPartition}
\begin{align}
&\text{for all } \xi \in \R^{n}, & \chi(\xi) + \sum_{j \geq 0} \varphi_{j}(\xi) &= 1, \label{eqn:DyadicPartition1} \\
&\text{for all } \xi \in \R^{n} \setminus \{ 0 \} , & \sum_{j \in \Z} \varphi_{j}(\xi) &= 1, \label{eqn:DyadicPartition2} \\
&\text{if } |j - j'| \geq 2, \text{ then } & \hspace{-0.5in} \supp \varphi_{j} \cap \supp \varphi_{j'} &= \varnothing, \label{eqn:DyadicPartition3} \\
&\text{if } j \geq 1, \text{ then } & \supp \chi \cap \supp \varphi_{j} &= \varnothing; \label{eqn:DyadicPartition4}
\intertext{the set $\widetilde{\mathcal{C}} := B(0, 2/3) + \mathcal{C}$ is an annulus, and}
&\text{if } |j - j'| \geq 5, \text{ then } & 2^{j'}\widetilde{\mathcal{C}} \cap 2^{j} \mathcal{C}  &= \varnothing. \label{eqn:DyadicPartition5}
\intertext{Furthermore, we have}
&\text{for all } \xi \in \R^{n}, & \frac{1}{2} \leq \chi^{2}(\xi) + \sum_{j \geq 0} \varphi_{j}^{2} (\xi) &\leq 1, \label{eqn:DyadicPartition6} \\
&\text{for all } \xi \in \R^{n} \setminus \{ 0 \} , & \frac{1}{2} \leq \sum_{j \in \Z} \varphi_{j}^{2} (\xi) &\leq 1. \label{eqn:DyadicPartition7}
\end{align}
\end{subequations}
Denote by
\[
\mathscr{F}[u](\xi) = \hat{u}(\xi) = \int_{\R^{n}} \ee^{-2\pi \ii x \cdot \xi} u(x) \, \rd x,
\]
the Fourier transform of $u$, and let $h = \mathscr{F}^{-1} \varphi$ and $\widetilde{h} = \mathscr{F}^{-1} \chi$. Given a measurable function $\sigma$ defined on $\R^{n}$ with at most polynomial growth at infinity, we define the Fourier multiplier operator $M_{\sigma}$ by $M_{\sigma} u := \mathcal{F}^{-1} (\sigma \hat{u})$.

For $j \in \Z$, the \emph{inhomogeneous dyadic blocks} $\Loc_{j}$ are defined as follows:
\begin{align*}
& \text{if } j \leq -2, \hspace{-0.5in} & \Loc_{j} u &= 0, \\
& & \Loc_{-1} u &= M_{\chi} u = \int_{\R^{n}} \widetilde{h}(y) u(x-y) \, \rd y, \\
& \text{if } j \geq 0, \hspace{-0.5in} & \Loc_{j} u &= M_{\varphi_{j}} u = 2^{jn} \int_{\R^{n}} h(2^{j} y) u(x-y) \, \rd y.
\end{align*}
The inhomogeneous low-frequency cut-off operator $S_{j}$ is defined by
\[
S_{j}u := \sum_{j' \leq j-1} \Loc_{j'} u.
\]
For $j \in \Z$, the \emph{homogeneous dyadic blocks} $\loc_{j}$ and the homogeneous low-frequency cut-off operator $\dot{S}_{j}$ are defined as follows:
\begin{align*}
\loc_{j} u &= M_{\varphi_{j}} u = 2^{jn} \int_{\R^{n}} h(2^{j} y) u(x-y) \, \rd y, \\
\dot{S}_{j} u &= M_{\chi_{j}} u = 2^{jn} \int_{\R^{n}} \widetilde{h}(2^{j} y) u(x-y) \, \rd y.
\end{align*}

Formally, we can write the following \emph{Littlewood--Paley decompositions}:
\[
\mathop{\mathrm{Id}} = \sum_{j \in \Z} \Loc_{j} \qquad \text{and} \qquad \mathop{\mathrm{Id}} = \sum_{j \in \Z} \loc_{j}.
\]
In the inhomogeneous case, the decomposition makes sense in $\Sch'(\R^{n})$: if $u \in \Sch'(\R^{n})$ is a tempered distribution, then $u = \lim_{j \to \infty} S_{j} u$ in $\Sch'(\R^{n})$. Unfortunately, the homogeneous case is a little more involved. We denote by $\Sch'_{h}(\R^{n})$ the space of tempered distributions such that
\[
\lim_{\lambda \to \infty} \norm{M_{\theta(\lambda \, \cdot \,)} u}_{L^{\infty}} = 0 \quad \text{ for any } \theta \in C^{\infty}_{c}(\R^{n}).
\]
Then the homogeneous decomposition makes sense in $\Sch'_{h}(\R^{n})$: if $u \in \Sch'_{h}(\R^{n})$, then $u = \lim_{j \to \infty} \dot{S}_{j} u$ in $\Sch'_{h}(\R^{n})$. Moreover, using the homogeneous decomposition, it is straightforward to show that
\[
\dot{S}_{j}u = \sum_{j' \leq j-1} \loc_{j'} u.
\]

Given a real number $s$ and two numbers $p, r \in [1, \infty]$, the \emph{homogeneous Besov space} $\dot{B}^{s}_{p,r}(\R^{n})$ consists of those distributions $u$ in $\Sch'_{h}(\R^{n})$ such that
\[
\norm{u}_{\dot{B}^{s}_{p,r}} := \bigg( \sum_{j \in \Z} 2^{rjs} \norm{\loc_{j} u}_{L^{p}}^{r} \bigg)^{1/r} < \infty
\]
if $r < \infty$, and
\[
\norm{u}_{\dot{B}^{s}_{p,\infty}} := \sup_{j \in \Z} 2^{js} \norm{\loc_{j} u}_{L^{p}} < \infty
\]
if $r = \infty$. This is a normed space, and its norm is independent of the choice of function $\varphi$ used to define the blocks $\loc_{j}$. Note that a distribution $u \in \Sch'_{h}(\R^{n})$ belongs to $\dot{B}^{s}_{p,r}(\R^{n})$ if, and only if, there exists a constant $C$ and a non-negative sequence $(d_{j})_{j \in \Z}$ such that
\begin{equation}
\label{eqn:StandardTrick}
\text{for all } j \in \Z, \quad \norm{\loc_{j} u}_{L^{p}} \leq C d_{j} 2^{-js} \qquad \text{and} \qquad \norm{(d_{j})}_{\ell^{r}} = 1.
\end{equation}
It follows immediately from \eqref{eqn:DyadicPartition7} that the seminorms $\norm{\cdot}_{\dot{H}^{s}}$ and $\norm{\cdot}_{\dot{B}^{s}_{2,2}}$ are equivalent, and hence that $\dot{H}^{s} \subset \dot{B}^{s}_{2,2}$ and that both spaces coincide for $s < n/2$.

We also define the \emph{inhomogeneous Besov space} $B^{s}_{p,r}(\R^{n})$ as the space of those distributions $u$ in $\Sch'(\R^{n})$ such that
\[
\norm{u}_{B^{s}_{p,r}} := \bigg( \sum_{j \in \Z} 2^{rjs} \norm{\Loc_{j} u}_{L^{p}}^{r} \bigg)^{1/r} < \infty
\]
if $r < \infty$, and
\[
\norm{u}_{B^{s}_{p,\infty}} := \sup_{j \in \Z} 2^{js} \norm{\Loc_{j} u}_{L^{p}} < \infty
\]
if $r = \infty$. It is straightforward to show that $B^{s}_{p,r} = \dot{B}^{s}_{p,r} \cap L^{p}$, and that $B^{s}_{p,r}$ is always a Banach space. For that reason, we focus mainly on homogeneous Besov spaces; most of the following results have inhomogeneous versions, which can be found in Sections~2.7 and 2.8 of \cite{book:BCD2011}.


\subsection{Embeddings}

Much like the Sobolev embeddings, Besov spaces embed in certain $L^{p}$ spaces with the correct exponents. We quote the two embeddings we will use most frequently.

\begin{prop}[Proposition~2.20 in \cite{book:BCD2011}]
\label{prop:BesovEmbedding}
Let $1 \leq p_{1} \leq p_{2} \leq \infty$ and $1 \leq r_{1} \leq r_{2} \leq \infty$. For any real number $s$, we have the continuous embedding
\[
\dot{B}^{s}_{p_{1}, r_{1}}(\R^{n}) \hookrightarrow \dot{B}^{s - n (1/p_{1} - 1/p_{2})}_{p_{2}, r_{2}}(\R^{n}).
\]
\end{prop}

\begin{prop}[Proposition~2.39 in \cite{book:BCD2011}]
\label{prop:BesovLebesgueEmbedding}
For $1 \leq p \leq q \leq \infty$, we have the continuous embedding
\[
\dot{B}^{n/p-n/q}_{p,1}(\R^{n}) \hookrightarrow L^{q}(\R^{n}).
\]
\end{prop}

Note that the homogeneous Besov space $\dot{B}^{s}_{p,r}(\R^{n})$ is a Banach space if, and only if, either $s < n/p$, or $s = n/p$ and $r = 1$ (in contrast to its inhomogeneous counterpart). Indeed, it is the case $\dot{B}^{n/p}_{p,1}$ that most interests us, especially when $p=2$, for three reasons: it is a Banach space, it embeds continuously in $L^{\infty}(\R^{n})$ by Proposition~\ref{prop:BesovLebesgueEmbedding}, and it is a Banach algebra. The last fact follows from Bony's paraproduct decomposition, which we outline now.

\subsection{Homogeneous Paradifferential Calculus}

Let $u$ and $v$ be tempered distributions in $\Sch'_{h}(\R^{n})$. We have
\[
u = \sum_{j' \in \Z} \loc_{j'} u \quad \text{and} \quad v = \sum_{j \in \Z} \loc_{j} v,
\]
so, at least formally,
\[
uv = \sum_{j, j' \in \Z} \loc_{j'} u \loc_{j} v.
\]
One of the key techniques of paradifferential calculus is to break the above sum into three parts, as follows: define
\[
\dot{T}_{u} v := \sum_{j \in \Z} \dot{S}_{j-1} u \loc_{j} v,
\]
and
\[
\dot{R}(u,v) := \sum_{|k-j| \leq 1} \loc_{k} u \loc_{j} v.
\]
At least formally, the following \emph{Bony decomposition} holds true:
\[
uv = \dot{T}_{u} v + \dot{T}_{v} u + \dot{R}(u,v).
\]
We now state two standard estimates on $\dot{T}$ and $\dot{R}$ that we will use in proving our a priori estimates in Section~\ref{sec:APriori2D3D}.

\begin{lemma}[Theorem 2.47 from \cite{book:BCD2011}]
\label{lem:BonyTEstimate}
Let $s \in \R$ and $t < 0$. There exists a constant $C = C(s,t)$ such that for any $p, r_{1}, r_{2} \in [1,\infty]$, $u \in \dot{B}^{t}_{p,r_{1}}$ and $v \in \dot{B}^{s}_{p,r_{2}}$,
\[
\norm{\dot{T}_{u} v}_{\dot{B}^{s+t}_{p,r}} \leq C \norm{u}_{\dot{B}^{t}_{\infty,r_{1}}} \norm{v}_{\dot{B}^{s}_{p,r_{2}}}
\]
with $\frac{1}{r} = \min \left\{ 1, \frac{1}{r_{1}} + \frac{1}{r_{2}} \right\}$.
\end{lemma}

\begin{lemma}[Theorem 2.52 from \cite{book:BCD2011}]
\label{lem:BonyREstimate}
Let $s_{1}, s_{2} \in \R$ such that $s_{1} + s_{2} > 0$. There exists a constant $C = C(s_{1}, s_{2})$ such that, for any $p_{1}, p_{2}, r_{1}, r_{2} \in [1,\infty]$, $u \in \dot{B}^{s_{1}}_{p_{1},r_{1}}$ and $v \in \dot{B}^{s_{2}}_{p_{2},r_{2}}$,
\[
\norm{\dot{R}(u,v)}_{\dot{B}^{s_{1}+s_{2}}_{p,r}} \leq C \norm{u}_{\dot{B}^{s_{1}}_{p_{1},r_{1}}} \norm{v}_{\dot{B}^{s_{2}}_{p_{2},r_{2}}}
\]
provided that
\[
\frac{1}{p} := \frac{1}{p_{1}} + \frac{1}{p_{2}} \leq 1 \quad \text{and} \quad \frac{1}{r} := \frac{1}{r_{1}} + \frac{1}{r_{2}} \leq 1.
\]
\end{lemma}

From Lemmas~\ref{lem:BonyTEstimate} and \ref{lem:BonyREstimate} it is straightforward to prove that, if $s > 0$ and $p, r \in [1, \infty]$ such that either $s < n/p$, or $s = n/p$ and $r=1$, then there is a constant $C$ depending only on $s$ and the dimension $n$ such that
\[
\norm{uv}_{\dot{B}^{s}_{p,r}} \leq C \left( \norm{u}_{L^{\infty}} \norm{v}_{\dot{B}^{s}_{p,r}} + \norm{u}_{\dot{B}^{s}_{p,r}} \norm{v}_{L^{\infty}} \right).
\]
In particular, $L^{\infty} \cap \dot{B}^{s}_{p,r}$ is a Banach algebra. Moreover, as $\dot{B}^{n/p}_{p,1}$ embeds continuously in $L^{\infty}$ (by Proposition~\ref{prop:BesovLebesgueEmbedding}), we see that $\dot{B}^{n/p}_{p,1}$ is an algebra and 
\begin{equation}
\label{eqn:BesovAlgebra}
\norm{uv}_{\dot{B}^{n/p}_{p,1}} \leq c \norm{u}_{\dot{B}^{n/p}_{p,1}} \norm{v}_{\dot{B}^{n/p}_{p,1}}.
\end{equation}

\section{A Priori Estimates}
%
\label{sec:APriori2D3D}

We first prove the two main a priori estimates that we will use in the existence proof: to streamline the presentation we prove the estimates formally for $\uu$ and $\BB$ which solve equations~\eqref{eqn:MHD-NonRes}.

\begin{prop}
\label{prop:B-Estimate}
If $(\uu, \BB)$ solve equations~\eqref{eqn:MHD-NonRes} on $[0,T]$, then there is a constant $c_{1}$ such that, for all $t \in [0,T]$,
\[
\norm{\BB(t)}_{\dot{B}^{n/2}_{2,1}} \leq \norm{\BB_{0}}_{\dot{B}^{n/2}_{2,1}} \exp \left( c_{1} \int_{0}^{t} \norm{\Grad \uu(s)}_{\dot{B}^{n/2}_{2,1}} \, \rd s \right).
\]
\end{prop}

Before embarking on the proof, we state a lemma we require, which is a particular case of Lemma~2.100 from \cite{book:BCD2011}.

\begin{lemma}
\label{lem:BesovCommutator}
Let $-1 - n/2 < \sigma < 1 + n/2$ and $1 \leq r \leq \infty$. Let $\vv$ be a divergence-free vector field on $\R^{n}$, and set $Q_{j} := [ (\vv \cdot \Grad), \loc_{j} ] f$. There exists a constant $C = C(\sigma, n)$, such that
\[
\bignorm{\left( 2^{j\sigma} \norm{Q_{j}}_{L^{2}} \right)_{j}}_{\ell^{r}} \leq C \norm{\Grad \vv}_{\dot{B}^{n/2}_{2, \infty} \cap L^{\infty}} \norm{f}_{\dot{B}^{\sigma}_{2,r}}.
\]
\end{lemma}

\begin{proof}[Proof of Proposition~\ref{prop:B-Estimate}]
Given $j \in \Z$, apply the homogeneous Littlewood--Paley operator $\loc_{j}$ (see Section~\ref{sec:BesovDefns}) to the equation \eqref{eqn:MHD-NonRes-B} for $\BB$ to obtain
\[
\frac{\pd}{\pd t} \loc_{j} \BB + \loc_{j} [(\uu \cdot \Grad) \BB] = \loc_{j} [(\BB \cdot \Grad) \uu].
\]
As $\dot{B}^{n/2}_{2,1}$ is an algebra (see equation~\eqref{eqn:BesovAlgebra}), we have
\[
\norm{(\BB \cdot \Grad) \uu}_{\dot{B}^{n/2}_{2,1}} \leq \norm{\BB}_{\dot{B}^{n/2}_{2,1}} \norm{\Grad \uu}_{\dot{B}^{n/2}_{2,1}}.
\]
By \eqref{eqn:StandardTrick}, we may write
\[
\norm{\loc_{j} [(\BB \cdot \Grad) \uu]}_{L^{2}} \leq C d_{j}(t) 2^{-jn/2} \norm{\BB}_{\dot{B}^{n/2}_{2,1}} \norm{\Grad \uu}_{\dot{B}^{n/2}_{2,1}}
\]
where $d_{j}(t)$ denotes a sequence in $\ell^{1}(\Z)$ whose sum is $1$.

For the term $(\uu \cdot \Grad) \BB$, we use Bony's paraproduct decomposition:
\[
(\uu \cdot \Grad) \BB_{\ell} = \sum_{k=1}^{n} [\dot{T}_{\uu_{k}} \pd_{k} \BB_{\ell} + \dot{T}_{\pd_{k} \BB_{\ell}} \uu_{k} + \dot{R} (\uu_{k}, \pd_{k} \BB_{\ell})].
\]
Consider the second term $\dot{T}_{\pd_{k} \BB_{\ell}} \uu_{k}$: by Lemma~\ref{lem:BonyTEstimate} we have
\begin{align*}
\norm{\dot{T}_{\pd_{k} \BB_{\ell}} \uu_{k}}_{\dot{B}^{n/2}_{2,1}} &\leq c \sum_{k=1}^{n} \norm{\pd_{k} \BB_{\ell}}_{\dot{B}^{-1}_{\infty,\infty}} \norm{\uu_{k}}_{\dot{B}^{n/2+1}_{2,1}} \\
&\leq c \norm{\BB}_{\dot{B}^{n/2}_{2,1}} \norm{\Grad \uu}_{\dot{B}^{n/2}_{2,1}},
\end{align*}
where we have used that $\dot{B}^{n/2}_{2,1} \hookrightarrow \dot{B}^{0}_{\infty,\infty}$ (by Proposition~\ref{prop:BesovEmbedding}). For the third term $\dot{R} (\uu_{k}, \pd_{k} \BB_{\ell})$, we apply Lemma~\ref{lem:BonyREstimate}:
\begin{align*}
\norm{\dot{R}(\uu_{k}, \pd_{k} \BB_{\ell})}_{\dot{B}^{n/2}_{2,1}} &\leq c \sum_{k=1}^{n} \norm{\uu_{k}}_{\dot{B}^{n/2+1}_{2,1}} \norm{\pd_{k} \BB_{\ell}}_{\dot{B}^{-1}_{\infty, \infty}} \\
&\leq c \norm{\Grad \uu}_{\dot{B}^{n/2}_{2,1}} \norm{\BB}_{\dot{B}^{n/2}_{2,1}},
\end{align*}
as above. Using \eqref{eqn:StandardTrick}, we obtain
\begin{align*}
\sum_{k=1}^{n} \norm{\loc_{j} \dot{T}_{\pd_{k} \BB_{\ell}} \uu_{k}}_{L^{2}} &\leq c d_{j}(t) 2^{-jn/2} \norm{\Grad \uu}_{\dot{B}^{n/2}_{2,1}} \norm{\BB}_{\dot{B}^{n/2}_{2,1}}, \\
\sum_{k=1}^{n} \norm{\loc_{j} \dot{R}(\uu_{k}, \pd_{k} \BB_{\ell})}_{L^{2}} &\leq c d_{j}(t) 2^{-jn/2} \norm{\Grad \uu}_{\dot{B}^{n/2}_{2,1}} \norm{\BB}_{\dot{B}^{n/2}_{2,1}}.
\end{align*}

For the term $\dot{T}_{\uu_{k}} \pd_{k} \BB_{\ell}$, let us write
\begin{align*}
\sum_{k=1}^{n} \loc_{j} \dot{T}_{\uu_{k}} \pd_{k} \BB_{\ell} &= \sum_{j' \in \Z} \sum_{k=1}^{n} \loc_{j} \left( \dot{S}_{j'-1} \uu_{k} \pd_{k} \loc_{j'} \BB_{\ell} \right) \\
&= \sum_{k=1}^{n} \dot{S}_{j-1} \uu_{k} \pd_{k} \loc_{j} \BB_{\ell} \\
&\qquad + \sum_{j' \in \Z} \sum_{k=1}^{n} (\dot{S}_{j'-1} \uu_{k} - \dot{S}_{j-1} \uu_{k}) \pd_{k} \loc_{j} \loc_{j'} \BB_{\ell} \\
&\qquad  + \sum_{j' \in \Z} \sum_{k=1}^{n} [\loc_{j}, \dot{S}_{j'-1} \uu_{k} \pd_{k}] \left( \loc_{j'} \BB_{\ell} \right) \\
&=: (\dot{S}_{j-1} \uu \cdot \Grad) \loc_{j} \BB_{\ell} + P_{j} + Q_{j}.
\end{align*}
For $P_{j}$, by \eqref{eqn:DyadicPartition3} we have
\begin{align*}
P_{j} :=& \sum_{|j - j'| \leq 1} \sum_{k=1}^{n} (\dot{S}_{j'-1} \uu_{k} - \dot{S}_{j-1} \uu_{k}) \loc_{j} \loc_{j'} \pd_{k} \BB_{\ell} \\
=& \sum_{k=1}^{n} (\loc_{j-1} \uu_{k}) (\loc_{j} \loc_{j+1} \pd_{k} \BB_{\ell}) - \sum_{k=1}^{n} (\loc_{j-2} \uu_{k}) (\loc_{j} \loc_{j-1} \pd_{k} \BB_{\ell}),
\end{align*}
so as $\norm{\loc_{j} \pd_{k} \BB}_{L^{2}} \simeq 2^{j} \norm{\loc_{j} \BB}_{L^{2}}$ we have
\begin{align*}
2^{jn/2} \norm{P_{j}}_{L^{2}} &\leq c \Big( 4 \cdot 2^{j-1} \norm{\loc_{j-1} \uu}_{L^{\infty}} 2^{jn/2} \norm{\loc_{j} \BB_{\ell}}_{L^{2}} \\ &\qquad + 2 \cdot 2^{j-2} \norm{\loc_{j-2} \uu}_{L^{\infty}} 2^{jn/2} \norm{\loc_{j} \BB_{\ell}}_{L^{2}} \Big)\\
&\leq c d_{j}(t) \norm{\uu}_{\dot{B}^{1}_{\infty, \infty}} \norm{\BB}_{\dot{B}^{n/2}_{2,1}} \\
&\leq c d_{j}(t) \norm{\Grad \uu}_{\dot{B}^{n/2}_{2,1}} \norm{\BB}_{\dot{B}^{n/2}_{2,1}}.
\end{align*}
For $Q_{j}$, we apply Lemma~\ref{lem:BesovCommutator}: note that
\[
Q_{j} := \sum_{j' \in \Z} [\loc_{j}, \dot{S}_{j'-1} (\uu \cdot \Grad)] \left( \loc_{j'} \BB_{\ell} \right)
\]
so
\begin{align*}
\bignorm{\left( 2^{jn/2} \norm{Q_{j}}_{L^{2}} \right)_{j}}_{\ell^{1}} &\leq c \norm{\Grad \uu}_{\dot{B}^{n/2}_{2, \infty} \cap L^{\infty}} \norm{\BB}_{\dot{B}^{n/2}_{2,1}} \\
&\leq c \norm{\Grad \uu}_{\dot{B}^{n/2}_{2, 1}} \norm{\BB}_{\dot{B}^{n/2}_{2,1}}
\end{align*}
since $\dot{B}^{n/2}_{2, 1}$ embeds continuously in both $\dot{B}^{n/2}_{2, \infty}$ (by Proposition~\ref{prop:BesovEmbedding}) and $L^{\infty}$ (by Proposition~\ref{prop:BesovLebesgueEmbedding}). So by \eqref{eqn:StandardTrick},
\[
\norm{Q_{j}}_{L^{2}} \leq c d_{j}(t) 2^{-jn/2} \norm{\Grad \uu}_{\dot{B}^{n/2}_{2, 1}} \norm{\BB}_{\dot{B}^{n/2}_{2,1}}.
\]

By combining all the above estimates, we obtain
\begin{equation}
\label{eqn:NewLocalisedEqn}
\frac{\pd}{\pd t} \loc_{j} \BB + (\dot{S}_{j-1} \uu \cdot \Grad) \loc_{j} \BB = F_{j}(t),
\end{equation}
where
\[
\norm{F_{j}(t)}_{L^{2}} \leq c d_{j}(t) 2^{-jn/2} \norm{\Grad \uu}_{\dot{B}^{n/2}_{2,1}} \norm{\BB}_{\dot{B}^{n/2}_{2,1}}.
\]
Taking the inner product of \eqref{eqn:NewLocalisedEqn} with $\loc_{j} \BB$ and using the fact that $\uu$ (and hence $\dot{S}_{j-1} \uu$) is divergence-free, we obtain
\[
2^{jn/2} \frac{\rd}{\rd t} \norm{\loc_{j} \BB}_{L^{2}} \leq 2c d_{j}(t) \norm{\Grad \uu}_{\dot{B}^{n/2}_{2,1}} \norm{\BB}_{\dot{B}^{n/2}_{2,1}}
\]
so summing in $j$ yields
\[
\frac{\rd}{\rd t} \norm{\BB}_{\dot{B}^{n/2}_{2,1}} \leq c \norm{\Grad \uu}_{\dot{B}^{n/2}_{2,1}} \norm{\BB}_{\dot{B}^{n/2}_{2,1}}
\]
and the result follows by Gronwall's inequality.
\end{proof}

Our second estimate, for the $\uu$ equation alone, is stated for a general forcing term $\ff$.

\begin{prop}
\label{prop:U-Estimate}
Let $\ff \in L^{1}(0, T; \dot{B}^{n/2-1}_{2,1}(\R^{n}))$. Suppose $\uu$ solves
\begin{subequations}
\label{eqn:NSE}
\begin{align}
\frac{\pd \uu}{\pd t} + (\uu \cdot \Grad) \uu - \nu \Laplace \uu + \Grad p &= \ff, \\
\Div \uu &= 0,
\end{align}
\end{subequations}
on the time interval $[0, T]$. Then there is a constant $c_{2}$ such that, for all $t \in [0, T]$,
\begin{align*}
&\norm{\uu(t)}_{\dot{B}^{n/2-1}_{2,1}} + \nu \int_{0}^{t} \norm{\Grad \uu(s)}_{\dot{B}^{n/2}_{2,1}} \, \rd s \\
&\qquad \qquad \leq \norm{\uu_{0}}_{\dot{B}^{n/2-1}_{2,1}} + c_{2} \int_{0}^{t} \norm{\uu(s)}_{H^{n/2}}^{2} \, \rd s + c_{2} \int_{0}^{t} \norm{\ff(s)}_{\dot{B}^{n/2-1}_{2,1}} \, \rd s.
\end{align*}
\end{prop}

Note that in the particular case $\ff = (\BB \cdot \Grad) \BB = \Div (\BB \otimes \BB)$, we have
\begin{equation}
\label{eqn:FtoB}
\norm{\ff}_{\dot{B}^{n/2-1}_{2,1}} = \norm{\Div (\BB \otimes \BB)}_{\dot{B}^{n/2-1}_{2,1}} \leq \norm{\BB}_{\dot{B}^{n/2}_{2,1}}^{2}
\end{equation}
since $\dot{B}^{n/2}_{2,1}$ is an algebra.

For the proof we will need the following lemma.

\begin{lemma}[Lemma 1.1 from \cite{art:Chemin1992}]
\label{lem:SobolevNonlinear}
Let $\vv$ be a divergence-free vector field, and let $s$, $r$, $r'$, $r''$ be four real numbers such that $r + r' + r'' = n/2 + 1 + 2s$, $r + r' > 0$, $0 \leq r < n/2 + 1$ and $r' < n/2 + 1$. Then there exists a constant $C$ such that
\[
\inner{\Lambda^{s} [(\vv \cdot \Grad) \ww]}{\Lambda^{s} \ww} \leq C \left( \norm{\vv}_{H^{r}} \norm{\ww}_{\dot{H}^{r'}} + \norm{\ww}_{H^{r}} \norm{\vv}_{\dot{H}^{r'}} \right) \norm{\ww}_{\dot{H}^{r''}}.
\]
\end{lemma}

In particular, taking $s = n/2 - 1$, $r = r' = n/2$ and $r'' = n/2 - 1$ yields
\begin{align}
&\inner{\Lambda^{n/2 - 1} [(\vv \cdot \Grad) \ww]}{\Lambda^{n/2 - 1} \ww} \notag\\
&\qquad \qquad \leq C \left( \norm{\vv}_{H^{n/2}} \norm{\ww}_{\dot{H}^{n/2}} + \norm{\ww}_{H^{n/2}} \norm{\vv}_{\dot{H}^{n/2}}  \right) \norm{\ww}_{\dot{H}^{n/2 - 1}}.\label{eqn:SobolevNonlinear}
\end{align}

\begin{proof}[Proof of Proposition~\ref{prop:U-Estimate}]
Applying the Littlewood--Paley operator $\loc_{j}$ to equation \eqref{eqn:NSE} yields
\[
\frac{\pd}{\pd t} \loc_{j} \uu + \loc_{j} [ (\uu \cdot \Grad) \uu] - \nu \Laplace \loc_{j} \uu + \Grad \loc_{j} p = \loc_{j} \ff.
\]
Taking the inner product with $\loc_{j} \uu$ yields
\[
\frac{1}{2} \frac{\rd}{\rd t} \norm{\loc_{j} \uu}_{L^{2}}^{2} + c \nu 2^{2j} \norm{\loc_{j} \uu}_{L^{2}}^{2} \leq \abs{\inner{\loc_{j} [ (\uu \cdot \Grad) \uu]}{\loc_{j} \uu}} + \abs{\inner{\loc_{j} \ff}{\loc_{j} \uu}}
\]

Applying the estimate from equation~\eqref{eqn:SobolevNonlinear} yields
\begin{align*}
\inner{\Lambda^{n/2-1} [(\uu \cdot \Grad) \uu]}{\Lambda^{n/2-1} \uu} &\leq c \norm{\uu}_{H^{n/2}} \norm{\uu}_{\dot{H}^{n/2}} \norm{\uu}_{\dot{H}^{n/2-1}} \\
&\leq c \norm{\uu}_{H^{n/2}}^{2} \norm{\uu}_{\dot{B}^{n/2-1}_{2,1}}.
\end{align*}
Decomposing each term on the left-hand side, we obtain
\[
\sum_{j, j' \in \Z} \inner{2^{j(n/2-1)} \loc_{j} [(\uu \cdot \Grad) \uu]}{2^{j'(n/2-1)} \loc_{j'} \uu} \leq c \norm{\uu}_{H^{n/2}}^{2} \norm{\uu}_{\dot{B}^{n/2-1}_{2,1}}
\]
Just taking the sum of the ``diagonal'' terms  yields
\[
\sum_{j \in \Z} 2^{j(n-2)} \inner{\loc_{j} [(\uu \cdot \Grad) \uu]}{\loc_{j} \uu} \leq c \norm{\uu}_{H^{n/2}}^{2} \norm{\uu}_{\dot{B}^{n/2-1}_{2,1}}.
\]
By \eqref{eqn:StandardTrick} (and dividing by $2^{j(n-2)}$) we obtain
\[
\abs{\inner{\loc_{j} [ (\uu \cdot \Grad) \uu]}{\loc_{j} \uu}} \leq c d_{j}(t) 2^{-j(n/2-1)} \norm{\uu}_{H^{n/2}}^{2} \norm{\loc_{j} \uu}_{L^{2}}.
\]
Hence
\begin{align*}
&\frac{\rd}{\rd t} \norm{\loc_{j} \uu(t)}_{L^{2}}^{2} + c \nu 2^{2j} \norm{\loc_{j} \uu(t)}_{L^{2}}^{2} \\
&\qquad \qquad \leq c d_{j}(t) 2^{-j(n/2-1)} \left( \norm{\uu(t)}_{H^{n/2}}^{2} + \norm{\ff(t)}_{\dot{B}^{n/2-1}_{2,1}} \right) \norm{\loc_{j} \uu(t)}_{L^{2}}.
\intertext{Dividing through by $\norm{\loc_{j} \uu(t)}_{L^{2}}$ and multiplying by $\ee^{c\nu 2^{2j}t}$ yields}
&\frac{\rd}{\rd t} \left( \ee^{c\nu 2^{2j}t} \norm{\loc_{j} \uu(t)}_{L^{2}} \right) \\
&\qquad \qquad \leq c \ee^{c\nu 2^{2j}t} d_{j}(t) 2^{-j(n/2-1)} \left( \norm{\uu(t)}_{H^{n/2}}^{2} + \norm{\ff(t)}_{\dot{B}^{n/2-1}_{2,1}} \right).
\end{align*}
Integrating in time from $0$ to $t$ yields
\begin{align}
&\norm{\loc_{j} \uu(t)}_{L^{2}} \leq \norm{\loc_{j} \uu_{0}}_{L^{2}} \ee^{-c\nu 2^{2j}t}  \label{eqn:IntegralU-chunk} \\
&\qquad + c 2^{-j(n/2-1)} \int_{0}^{t} d_{j}(s) \ee^{-c\nu 2^{2j}(t-s)} \left( \norm{\uu(s)}_{H^{n/2}}^{2} + \norm{\ff(s)}_{\dot{B}^{n/2-1}_{2,1}} \right) \, \rd s. \notag
\end{align}

As $\ee^{-c\nu 2^{2j}t} \leq 1$ for all $t$, multiplying \eqref{eqn:IntegralU-chunk} by $2^{j(n/2-1)}$ and summing in $j$ yields
\[
\norm{\uu(t)}_{\dot{B}^{n/2-1}_{2,1}} \leq \norm{\uu_{0}}_{\dot{B}^{n/2-1}_{2,1}} + c \int_{0}^{t} \left( \norm{\uu(s)}_{H^{n/2}}^{2} + \norm{\ff(s)}_{\dot{B}^{n/2-1}_{2,1}} \right) \, \rd s.
\]
Taking the $L^{\infty}$ norm over $t \in [0,T]$ yields
\[
\norm{\uu}_{L^{\infty}(0, T; \dot{B}^{n/2-1}_{2,1})} \leq \norm{\uu_{0}}_{\dot{B}^{n/2-1}_{2,1}} + c \int_{0}^{T} \left( \norm{\uu(t)}_{H^{n/2}}^{2} + \norm{\ff(t)}_{\dot{B}^{n/2-1}_{2,1}} \right) \, \rd t.
\]

Multiplying \eqref{eqn:IntegralU-chunk} by $\nu 2^{j(n/2+1)}$ and then taking the $L^{1}$ norm over $t \in [0,T]$ yields
\begin{align*}
&\nu2^{j(n/2)} \norm{\loc_{j} \Grad \uu}_{L^{1}(0, T; L^{2})} \leq 2^{j(n/2-1)} \norm{\loc_{j} \uu_{0}}_{L^{2}} \int_{0}^{T} \nu 2^{2j} \ee^{-c\nu 2^{2j}t} \, \rd t \\
&\qquad + c \int_{0}^{T} \int_{0}^{t} d_{j}(s) \nu 2^{2j} \ee^{-c\nu 2^{2j}(t-s)} \left( \norm{\uu(s)}_{H^{n/2}}^{2} + \norm{\ff(s)}_{\dot{B}^{n/2-1}_{2,1}} \right) \, \rd s \, \rd t.
\end{align*}
Using Young's inequality for convolutions and the fact that
\[
\int_{0}^{T} c \nu 2^{2j} \ee^{-c\nu 2^{2j}t} \, \rd t = 1 - \ee^{-c\nu 2^{2j}T} \leq 1
\]
yields
\begin{align*}
&\nu 2^{j(n/2)} \norm{\loc_{j} \Grad \uu}_{L^{1}(0, T; L^{2})} \leq c 2^{j(n/2-1)} \norm{\loc_{j} \uu_{0}}_{L^{2}} \\
&\qquad \qquad + c \int_{0}^{T} d_{j}(t) \left( \norm{\uu(t)}_{H^{n/2}}^{2} + \norm{\ff(t)}_{\dot{B}^{n/2-1}_{2,1}} \right) \, \rd t.
\end{align*}
Summation in $j$ and the Monotone Convergence Theorem yields
\[
\nu \norm{\Grad \uu}_{L^{1}(0, T; \dot{B}^{n/2}_{2,1})} \leq \norm{\uu_{0}}_{\dot{B}^{n/2-1}_{2,1}} + c \int_{0}^{T} \left( \norm{\uu(t)}_{H^{n/2}}^{2} + \norm{\ff(t)}_{\dot{B}^{n/2-1}_{2,1}} \right) \, \rd t.
\]
This completes the proof.
\end{proof}

\section{Uniform Bounds in 2D and 3D}
\label{sec:Bounds}

To turn our a priori estimates into a rigorous proof, we consider a Fourier truncation of the equations~\eqref{eqn:MHD-NonRes}. We define the Fourier truncation $\Ss_{R}$ as follows:
\[
\widehat{\Ss_{R} f}(\xi) = \mathbbm{1}_{B_{R}}(\xi) \hat{f}(\xi),
\]
where $B_{R}$ denotes the ball of radius $R$ centered at the origin. Note that
\begin{align*}
\norm{\Ss_{R}f - f}_{H^{s}}^{2} &= \int_{(B_{R})^{c}} (1 + |\xi|^{2})^{s} |\hat{f}(\xi)|^{2} \, \rd \xi \\
&= \int_{(B_{R})^{c}} \frac{1}{(1 + |\xi|^{2})^{k}} (1 + |\xi|^{2})^{s+k} |\hat{f}(\xi)|^{2} \, \rd \xi \\
&\leq \frac{1}{(1 + R^{2})^{k}} \int_{(B_{R})^{c}} (1 + |\xi|^{2})^{s+k} |\hat{f}(\xi)|^{2} \, \rd \xi \\
&\leq \frac{C}{R^{2k}} \norm{f}_{H^{s+k}}^{2}.
\end{align*}
Hence
\begin{align}
\norm{\Ss_{R} f - f}_{H^{s}} &\leq C (1/R)^{k} \norm{f}_{H^{s+k}}, \label{eqn:MollifierProp1}\\
\norm{\Ss_{R} f - \Ss_{R'} f}_{H^{s}} &\leq C \max\{(1/R)^{k}, (1/R')^{k}\} \norm{f}_{H^{s+k}}. \label{eqn:MollifierProp2}
\end{align}

We consider the truncated MHD equations on the whole of $\R^{n}$:
\begin{subequations}
\label{eqn:MHD-Cutoff-Besov}
\begin{align}
\frac{\pd \uu^{R}}{\pd t} - \nu \Laplace \uu^{R} +  \Grad p_{*}^{R} &= \Ss_{R}[(\BB^{R} \cdot \Grad) \BB^{R}] - \Ss_{R}[(\uu^{R} \cdot \Grad) \uu^{R}], \label{eqn:MHD-Cutoff-Besov-u} \\
\frac{\pd \BB^{R}}{\pd t} &= \Ss_{R}[(\BB^{R} \cdot \Grad) \uu^{R}] - \Ss_{R}[(\uu^{R} \cdot \Grad) \BB^{R}], \label{eqn:MHD-Cutoff-Besov-B} \\
\Div \uu^{R} &= \Div \BB^{R} = 0,
\end{align}
\end{subequations}
with initial data $\Ss_{R} \uu_{0}, \Ss_{R} \BB_{0}$. By taking the truncated initial data as we have, we ensure that $\uu^{R}, \BB^{R}$ lie in the space
\[
V_{R} := \{ f \in L^{2}(\R^{n}) : \hat{f} \text{ is supported in } B_{R} \},
\]
as the truncations are invariant under the flow of the equations. The Fourier truncations act like mollifiers, smoothing the equation; in particular, on the space $V_{R}$ it is easy to show that
\[
F(\uu^{R}, \BB^{R}) := \Ss_{R}[(\uu^{R} \cdot \Grad) \BB^{R}]
\]
is Lipschitz in $\uu^{R}$ and $\BB^{R}$. Hence, by Picard's theorem for infinite-dimensional ODEs (see Theorem~3.1 in \cite{book:MajdaBertozzi}, for example), there exists a solution $(\uu^{R}, \BB^{R})$ in $V_{R}$ to \eqref{eqn:MHD-Cutoff-Besov} for some  time interval $[0, T(R)]$.

The solution will exist as long as the relevant norms of $\uu^{R}$ and $\BB^{R}$ remain finite. Repeating the a priori estimates from Proposition~\ref{prop:B-Estimate} we obtain
\[
\norm{\BB^{R}(t)}_{\dot{B}^{n/2}_{2,1}} \leq \norm{\BB_{0}}_{\dot{B}^{n/2}_{2,1}} \exp \left( c_{1} \int_{0}^{t} \norm{\Grad \uu^{R}(s)}_{\dot{B}^{n/2}_{2,1}} \, \rd s \right),
\]
where the constant $c_{1}$ is independent of $R$. Repeating Proposition~\ref{prop:U-Estimate} for the equation
\begin{subequations}
\label{eqn:NSE-Cutoff}
\begin{align}
\frac{\pd \uu^{R}}{\pd t} + \Ss_{R}[(\uu^{R} \cdot \Grad) \uu^{R}] - \nu \Laplace \uu^{R} +  \Grad p_{*}^{R} &= \ff^{R}, \\
\Div \uu^{R} &= 0.
\end{align}
\end{subequations}
yields
\begin{align*}
&\norm{\uu^{R}(t)}_{\dot{B}^{n/2-1}_{2,1}} + \nu \int_{0}^{t} \norm{\Grad \uu^{R}(s)}_{\dot{B}^{n/2}_{2,1}} \, \rd s \\
&\qquad \qquad \leq \norm{\uu_{0}}_{\dot{B}^{n/2-1}_{2,1}} + c_{2} \int_{0}^{t} \norm{\uu^{R}(s)}_{H^{n/2}}^{2} \, \rd s + c_{2} \int_{0}^{t} \norm{\ff^{R}(s)}_{\dot{B}^{n/2-1}_{2,1}} \, \rd s,
\end{align*}
where the constant $c_{2}$ is independent of $R$.

Turning these estimates into uniform bounds on $\uu^{R}$ and $\BB^{R}$ which are independent of $R$ depends on the dimension, so we consider the 2D and 3D cases separately. However, in both cases we will make use of the following standard energy estimate:
\begin{align}
&\sup_{t \in [0, T]} \norm{\uu^{R}(t)}_{L^{2}}^{2} + \sup_{t \in [0, T]} \norm{\BB^{R}(t)}_{L^{2}}^{2} + \nu \int_{0}^{T} \norm{\Grad \uu^{R}(s)}_{L^{2}}^{2} \, \rd s \notag \\
&\qquad \qquad \qquad \qquad \leq 2 (\norm{\uu_{0}}_{L^{2}}^{2} + \norm{\BB_{0}}_{L^{2}}^{2}) \label{eqn:Energy}
\end{align}
for any $T > 0$, which can be obtained by taking the inner product of \eqref{eqn:MHD-NonRes-u} with $\uu^{R}$, the inner product of \eqref{eqn:MHD-NonRes-B} with $\BB^{R}$, and adding.

\subsection{Uniform Bounds in Two Dimensions}
\label{sec:Besov2D}

In 2D, the term $\int_{0}^{t} \norm{\uu^{R}(s)}_{H^{n/2}}^{2} \, \rd s$ is simply $\int_{0}^{t} \norm{\uu^{R}(s)}_{H^{1}}^{2} \, \rd s$. Using the standard energy estimate~\eqref{eqn:Energy} we may bound this as follows:
\begin{align}
\int_{0}^{t} \norm{\uu^{R}(s)}_{H^{1}}^{2} \, \rd s &\leq \int_{0}^{t} \norm{\uu^{R}(s)}_{L^{2}}^{2} \, \rd s + \int_{0}^{t} \norm{\Grad \uu^{R}(s)}_{L^{2}}^{2} \, \rd s \notag \\
&\leq 2 \left( t + \frac{1}{\nu} \right) (\norm{\uu_{0}}_{L^{2}}^{2} + \norm{\BB_{0}}_{L^{2}}^{2}). \label{eqn:Besov2DEnergy}
\end{align}

Using this, we show that $\uu^{R}$ and $\BB^{R}$ are uniformly bounded.

\begin{thm}
\label{thm:UniformBesov2D}
Let $n=2$, and let $(\uu^{R}, \BB^{R})$ be the solution to \eqref{eqn:MHD-Cutoff-Besov}. There is a time $T_{*} = T_{*} (\nu, \norm{\uu_{0}}_{B^{0}_{2,1}}, \norm{\BB_{0}}_{B^{1}_{2,1}}) > 0$ such that
\begin{align*}
\uu^{R} \text{ is uniformly bounded in } & L^{\infty}(0, T_{*}; \dot{B}^{0}_{2,1}(\R^{2})) \cap L^{1}(0, T_{*}; \dot{B}^{2}_{2,1}(\R^{2})), \\
\BB^{R} \text{ is uniformly bounded in } & L^{\infty}(0, T_{*}; \dot{B}^{1}_{2,1}(\R^{2})).
\end{align*}
\end{thm}


\begin{proof}
Let
\begin{align*}
M_{1} &= \norm{\uu_{0}}_{\dot{B}^{0}_{2,1}} + \frac{2c_{2}}{\nu} (\norm{\uu_{0}}_{L^{2}}^{2} + \norm{\BB_{0}}_{L^{2}}^{2}), \\
M_{2} &= 2c_{2} (\norm{\uu_{0}}_{L^{2}}^{2} + \norm{\BB_{0}}_{L^{2}}^{2}).
\end{align*}
Substituting from equation~\eqref{eqn:Besov2DEnergy} into Proposition~\ref{prop:U-Estimate}, we obtain
\[
\norm{\uu^{R}(t)}_{\dot{B}^{0}_{2,1}} + \nu \int_{0}^{t} \norm{\Grad \uu^{R}(s)}_{\dot{B}^{1}_{2,1}} \, \rd s \leq M_{1} + M_{2} t + c_{2} \int_{0}^{t} \norm{\ff^{R}(s)}_{\dot{B}^{0}_{2,1}} \, \rd s.
\]
Using \eqref{eqn:FtoB} and substituting in from Proposition~\ref{prop:B-Estimate}, we obtain
\begin{align*}
&\norm{\uu^{R}(t)}_{\dot{B}^{0}_{2,1}} + \nu \int_{0}^{t} \norm{\Grad \uu^{R}(s)}_{\dot{B}^{1}_{2,1}} \, \rd s \\
&\qquad \qquad \leq M_{1} + M_{2} t + \int_{0}^{t} \norm{\BB_{0}}_{\dot{B}^{1}_{2,1}}^{2} c_{2} \exp \left( 2c_{1} \int_{0}^{\tau} \norm{\Grad \uu^{R}(s)}_{\dot{B}^{1}_{2,1}} \, \rd s \right) \, \rd \tau. \\
&\qquad \qquad \leq M_{1} + M_{2} t + M_{3} t \exp \left( 2c_{1} \int_{0}^{t} \norm{\Grad \uu^{R}(s)}_{\dot{B}^{1}_{2,1}} \, \rd s \right),
\end{align*}
where $M_{3} = c_{1} \norm{\BB_{0}}_{\dot{B}^{1}_{2,1}}^{2}$. Let 
\begin{align*}
X_{R}(t) &= \norm{\uu^{R}(t)}_{\dot{B}^{0}_{2,1}}, \\
Y_{R}(t) &= \nu \int_{0}^{t} \norm{\Grad \uu^{R}(s)}_{\dot{B}^{1}_{2,1}} \, \rd s.
\end{align*}
Then we can rewrite the last inequality as
\begin{equation}
\label{eqn:ODE-XY-2D}
X_{R}(t) + Y_{R}(t) \leq M_{1} + M_{2} t + M_{3} t \exp ( 2 c_{1} Y_{R}(t) / \nu ).
\end{equation}

Set
\[
T_{*} = \min \left\{ \frac{M_{1}}{M_{2}}, \frac{M_{1}}{M_{3}} \exp (-6 c_{1} M_{1}/\nu) \right\}.
\]
It remains to show that $X_{R}(t) + Y_{R}(t) \leq 3 M_{1}$ for all $t \in [0, T_{*}]$ and all $R > 0$. To that end, note that $Y_{R}(t)$ is continuous and $Y_{R}(0) = 0$. Now, suppose $t < T_{*}$ and $Y_{R}(t) \leq 3 M_{1}$; then
\begin{align*}
Y_{R}(t) &\leq M_{1} + M_{2} t + M_{3} t \exp ( 2 c_{1} Y_{R}(t) / \nu ) \\
&< M_{1} + M_{1} + M_{1} \exp \left( \frac{2 c_{1}}{\nu} [ Y_{R}(t) - 3 M_{1}] \right) \\
&\leq 3 M_{1}.
\end{align*}
This means that $Y_{R}(t)$ can never equal $3 M_{1}$ on the interval $[0, T_{*})$; so $Y_{R}(t) < 3 M_{1}$ for all $t \in [0, T_{*})$. The result follows from inequality \eqref{eqn:ODE-XY-2D} and Proposition~\ref{prop:B-Estimate}.
\end{proof}

Before moving onto the 3D case, it is worth noting that in 2D the existence time $T_{*}$ depends only on the norm $\norm{\uu_{0}}_{B^{0}_{2,1}}$ rather than $\uu_{0}$ itself.


%


\subsection{Uniform Bounds in Three Dimensions}
\label{sec:Besov3D}

In 3D, we take initial data $\uu_{0} \in B^{1/2}_{2,1}(\R^{3})$ and $\BB_{0} \in B^{3/2}_{2,1}(\R^{3})$. Instead of being able to use the energy inequality, we require the following auxiliary estimate to bound $\int_{0}^{t} \norm{\Grad \uu^{R}(s)}_{\dot{H}^{1/2}}^{2} \, \rd s$.

\begin{prop}
\label{prop:AuxEstimate}
Let $n=3$. There exist constants $c_{3}$ and $c_{4}$ and a time $T_{1} = T_{1}(\nu, \uu_{0})$ such that, if $T \leq T_{1}$, $R > 0$ and
\begin{equation}
\label{eqn:Assumption}
\int_{0}^{T} \norm{\BB^{R}(s)}_{\dot{B}^{3/2}_{2,1}}^{2} \, \rd s \leq \frac{\nu}{2 (c_{3} c_{4})^{1/4}} =: C_{*},
\end{equation}
the solution $(\uu^{R}, \BB^{R})$ of \eqref{eqn:MHD-Cutoff-Besov} satisfies
\begin{align}
\int_{0}^{T} \norm{\Grad \uu^{R}(s)}_{\dot{H}^{1/2}}^{2} \, \rd s &\leq \frac{1}{\nu} \norm{\uu_{0}}_{\dot{H}^{1/2}}^{2} + \frac{8c_{3}}{\nu^{3}} \norm{\uu_{0}}_{\dot{H}^{1/2}}^{4} \notag \\
&\hspace{-24pt}+ \frac{3}{2\nu} \left( \int_{0}^{T} \norm{\BB^{R}(s)}_{\dot{B}^{3/2}_{2,1}}^{2} \, \rd s \right)^{2} + \frac{4c_{3}}{\nu^{3}}  \left( \int_{0}^{T} \norm{\BB^{R}(s)}_{\dot{B}^{3/2}_{2,1}}^{2} \, \rd s \right)^{4}. \label{eqn:AuxEstimate}
\end{align}
\end{prop}

Note carefully that the estimate~\eqref{eqn:AuxEstimate} is conditional on assumption~\eqref{eqn:Assumption} holding: once we have proved the proposition, we will require a further lemma to ensure that there is a time such that assumption~\eqref{eqn:Assumption} holds, and thus avoid a circular argument.

%
%

\begin{proof}[Proof of Proposition~\ref{prop:AuxEstimate}]
The proof is based on the proof of Theorem~1 in \cite*{art:MRS2013}, which in turn is based on the proof of Theorem~3.4 in \cite{book:CDGG}; the original idea of splitting the equation is due to \cite{art:Calderon1990}.

First, let us consider the Stokes equation with initial data $\uu_{0}$:
\begin{subequations}
\label{eqn:Split-H}
\begin{align}
\frac{\pd \hh}{\pd t} - \nu \Laplace \hh + \Grad p_{\hh} &= 0, \label{eqn:Split-H-Main} \\
\Div \hh &= 0, \label{eqn:Split-H-Div} \\
\hh(0) &= \uu_{0}, \label{eqn:Split-H-Init}
\end{align}
\end{subequations}
Thanks to the properties of the Stokes equation and of Fourier truncations, the solution of the equation
\begin{subequations}
\label{eqn:Split-HR}
\begin{align}
\frac{\pd \hh^{R}}{\pd t} - \nu \Laplace \hh^{R} + \Grad p_{\hh}^{R} &= 0, \label{eqn:Split-HR-Main}\\
\Div \hh^{R} &= 0, \label{eqn:Split-HR-Div}\\
\hh^{R}(0) &= \Ss_{R} \uu_{0}, \label{eqn:Split-HR-Init}
\end{align}
\end{subequations}
is given by $\hh^{R} = \Ss_{R} \hh$.

Let us decompose $\uu^{R} = \hh^{R} + \vv^{R} + \ww^{R}$, where $\vv^{R}$ and $\ww^{R}$ satisfy
\begin{subequations}
\label{eqn:Split-VR}
\begin{align}
\frac{\pd \vv^{R}}{\pd t} - \nu \Laplace \vv^{R} + \Grad p_{\vv}^{R} &= \Ss_{R} [(\BB^{R} \cdot \Grad) \BB^{R}], \label{eqn:Split-VR-Main}\\
\Div \vv^{R} &= 0, \label{eqn:Split-VR-Div}\\
\vv^{R}(0) &= 0, \label{eqn:Split-VR-Init}
\end{align}
\end{subequations}
and
\begin{subequations}
\label{eqn:Split-WR}
\begin{align}
\frac{\pd \ww^{R}}{\pd t} - \nu \Laplace \ww^{R} + \Ss_{R} [(\uu^{R} \cdot \Grad) \uu^{R}] + \Grad p_{\ww}^{R} &= 0, \label{eqn:Split-WR-Main}\\
\Div \ww^{R} &= 0, \label{eqn:Split-WR-Div}\\
\ww^{R}(0) &= 0, \label{eqn:Split-WR-Init}
\end{align}
\end{subequations}
respectively.

Applying $\Lambda^{1/2}$ to \eqref{eqn:Split-WR-Main} and taking the inner product with $\Lambda^{1/2} \ww^{R}$ yields
\begin{align*}
\frac{1}{2} \frac{\rd}{\rd t} \norm{\ww^{R}}_{\dot{H}^{1/2}}^{2} + \nu \norm{\ww^{R}}_{\dot{H}^{3/2}}^{2} &= \inner{\Lambda^{1/2} \Ss_{R} [ (\uu^{R} \cdot \Grad) \uu^{R}] }{\Lambda^{1/2} \ww^{R}} \\
&= \inner{(\uu^{R} \cdot \Grad) \uu^{R}}{\Lambda \ww^{R}} \\
&\leq \norm{\uu^{R}}_{L^{6}} \norm{\Grad \uu^{R}}_{L^{2}} \norm{\Lambda \ww^{R}}_{L^{3}} \\
&\leq c \norm{\uu^{R}}_{\dot{H}^{1}}^{2} \norm{\ww^{R}}_{\dot{H}^{3/2}} \\
&\leq c \left( \norm{\hh^{R}}_{\dot{H}^{1}}^{2} + \norm{\vv^{R}}_{\dot{H}^{1}}^{2} + \norm{\ww^{R}}_{\dot{H}^{1}}^{2} \right) \norm{\ww^{R}}_{\dot{H}^{3/2}} \\
&\leq c \norm{\hh^{R}}_{\dot{H}^{1}}^{2} \norm{\ww^{R}}_{\dot{H}^{3/2}} + c \norm{\vv^{R}}_{\dot{H}^{1}}^{2} \norm{\ww^{R}}_{\dot{H}^{3/2}} \\
&\qquad \qquad + c \norm{\ww^{R}}_{\dot{H}^{1/2}} \norm{\ww^{R}}_{\dot{H}^{3/2}}^{2}
\end{align*}
by interpolation. Using Young's inequality, we obtain
\[
\frac{\rd}{\rd t} \norm{\ww^{R}}_{\dot{H}^{1/2}}^{2} + \nu \norm{\ww^{R}}_{\dot{H}^{3/2}}^{2} \leq \frac{c_{3}}{\nu} \norm{\hh^{R}}_{\dot{H}^{1}}^{4} + \frac{c_{3}}{\nu} \norm{\vv^{R}}_{\dot{H}^{1}}^{4} + c_{4} \norm{\ww^{R}}_{\dot{H}^{1/2}} \norm{\ww}_{\dot{H}^{3/2}}^{2}.
\]
For any $T > t > 0$, integrating in time over $[0,t]$ yields
\begin{align*}
&\norm{\ww^{R}(t)}_{\dot{H}^{1/2}}^{2} + \nu \int_{0}^{t} \norm{\ww^{R}(s)}_{\dot{H}^{3/2}}^{2} \, \rd s \\
&\qquad \leq \frac{c_{3}}{\nu} \int_{0}^{t} \norm{\hh^{R}(s)}_{\dot{H}^{1}}^{4} \, \rd s + \frac{c_{3}}{\nu} \int_{0}^{t} \norm{\vv^{R}(s)}_{\dot{H}^{1}}^{4} \, \rd s \\
&\qquad \qquad \qquad + c_{4} \int_{0}^{t} \norm{\ww^{R}(s)}_{\dot{H}^{1/2}} \norm{\ww^{R}(s)}_{\dot{H}^{3/2}}^{2} \, \rd s \\
&\qquad \leq \frac{c_{3}}{\nu} \int_{0}^{T} \norm{\hh^{R}(s)}_{\dot{H}^{1}}^{4} \, \rd s + \frac{c_{3}}{\nu} \int_{0}^{T} \norm{\vv^{R}(s)}_{\dot{H}^{1}}^{4} \, \rd s \\
&\qquad \qquad \qquad + \frac{1}{2} \sup_{s \in [0,T]} \norm{\ww^{R}(s)}_{\dot{H}^{1/2}}^{2} + \frac{c_{4}}{2} \left( \int_{0}^{T} \norm{\ww^{R}(s)}_{\dot{H}^{3/2}}^{2} \, \rd s \right)^{2},
\end{align*}
so taking the supremum on the left-hand side over $t \in [0,T]$ yields
\begin{align}
&\sup_{s \in [0,T]} \norm{\ww^{R}(s)}_{\dot{H}^{1/2}}^{2} + 2\nu \int_{0}^{T} \norm{\ww^{R}(s)}_{\dot{H}^{3/2}}^{2} \, \rd s \notag \\
&\qquad \leq \frac{4c_{3}}{\nu} \int_{0}^{T} \norm{\hh^{R}(s)}_{\dot{H}^{1}}^{4} \, \rd s + \frac{4c_{3}}{\nu} \int_{0}^{T} \norm{\vv^{R}(s)}_{\dot{H}^{1}}^{4} \, \rd s \notag \\
&\qquad \qquad \qquad + 2c_{4} \left( \int_{0}^{T} \norm{\ww^{R}(s)}_{\dot{H}^{3/2}}^{2} \, \rd s \right)^{2}. \label{eqn:WEstimate}
\end{align}

Set
\[
T(R) := \sup \left \{ T \geq 0 : \int_{0}^{T} \norm{\ww^{R}(s)}_{\dot{H}^{3/2}}^{2} \, \rd s \leq \frac{\nu}{2c_{4}} \right\}
\]
so that for all $T \in [0, T(R)]$ we have
\begin{align}
&\sup_{s \in [0,T]} \norm{\ww^{R}(s)}_{\dot{H}^{1/2}}^{2} +  \nu \int_{0}^{T} \norm{\ww^{R}(s)}_{\dot{H}^{3/2}}^{2} \, \rd s \notag\\
&\qquad \leq \frac{4c_{3}}{\nu} \int_{0}^{T} \norm{\hh^{R}(s)}_{\dot{H}^{1}}^{4} \, \rd s + \frac{4c_{3}}{\nu} \int_{0}^{T} \norm{\vv^{R}(s)}_{\dot{H}^{1}}^{4} \, \rd s. \label{eqn:WEstimate2}
\end{align}
We now seek a bound on the right-hand side: indeed, if we can find a time $T_{0}$ such that
\begin{equation}
\label{eqn:DesiredBound}
\frac{4c_{3}}{\nu} \int_{0}^{T_{0}} \norm{\hh^{R}(s)}_{\dot{H}^{1}}^{4} \, \rd s + \frac{4c_{3}}{\nu} \int_{0}^{T_{0}} \norm{\vv^{R}(s)}_{\dot{H}^{1}}^{4} \, \rd s < \frac{\nu^{2}}{2c_{4}},
\end{equation}
then $T(R) \geq T_{0}$. To see this, we proceed along the same lines as in the proof of Theorem~\ref{thm:UniformBesov2D}: first note that
\[
\int_{0}^{T(R)} \norm{\ww^{R}(s)}_{\dot{H}^{3/2}}^{2} \, \rd s = \frac{\nu}{2c_{4}}
\]
by continuity; but if $T(R) < T_{0}$ then \eqref{eqn:WEstimate} and \eqref{eqn:DesiredBound} would imply that
\[
\int_{0}^{T(R)} \norm{\ww^{R}(s)}_{\dot{H}^{3/2}}^{2} \, \rd s < \frac{\nu}{2c_{4}},
\]
which is a contradiction, and thus we must have $T(R) \geq T_{0}$. 

First, let us find a bound for the $\hh$ term. Applying $\Lambda^{1/2}$ to \eqref{eqn:Split-H-Main} and taking the inner product with $\Lambda^{1/2} \hh$ yields
\[
\frac{1}{2} \frac{\rd}{\rd t} \norm{\hh}_{\dot{H}^{1/2}}^{2} + \nu \norm{\hh}_{\dot{H}^{3/2}}^{2} \leq 0.
\]
For any $T > t > 0$, integrating in time over $[0,t]$ yields
\[
\frac{1}{2} \norm{\hh(t)}_{\dot{H}^{1/2}}^{2} + \nu \int_{0}^{t} \norm{\hh(s)}_{\dot{H}^{3/2}}^{2} \, \rd s \leq \frac{1}{2} \norm{\uu_{0}}_{\dot{H}^{1/2}}^{2},
\]
and thus
\begin{equation}
\label{eqn:HEstimate}
\sup_{s \in [0,T]} \norm{\hh(s)}_{\dot{H}^{1/2}}^{2} + 2\nu \int_{0}^{T} \norm{\hh(s)}_{\dot{H}^{3/2}}^{2} \, \rd s \leq 2 \norm{\uu_{0}}_{\dot{H}^{1/2}}^{2}.
\end{equation}
By interpolation,
\begin{equation}
\label{eqn:HEstimate2}
\int_{0}^{T} \norm{\hh(s)}_{\dot{H}^{1}}^{4} \, \rd s \leq \frac{2}{\nu} \norm{\uu_{0}}_{\dot{H}^{1/2}}^{4}.
\end{equation}
Hence $\norm{\hh(s)}_{\dot{H}^{1}}^{4}$ is integrable on $[0,T]$, and thus we may choose $T_{1}$ such that
\[
\int_{0}^{T_{1}} \norm{\hh(s)}_{\dot{H}^{1}}^{4} \, \rd s < \frac{\nu^{3}}{16c_{3}c_{4}}.
\]
By the properties of the Stokes equation, this implies that
\begin{equation}
\label{eqn:HEstimateTime}
\int_{0}^{T_{1}} \norm{\hh^{R}(s)}_{\dot{H}^{1}}^{4} \, \rd s < \frac{\nu^{3}}{16c_{3}c_{4}}
\end{equation}
for all $R > 0$. Note that, unlike the 2D case, $T_{1}$ really depends on the whole of $\uu_{0}$.

Secondly, let us find a bound for the $\vv$ term. Applying $\Lambda^{1/2}$ to \eqref{eqn:Split-VR-Main} and taking the inner product with $\Lambda^{1/2} \vv^{R}$ yields
\begin{align*}
\frac{1}{2} \frac{\rd}{\rd t} \norm{\vv^{R}}_{\dot{H}^{1/2}}^{2} + \nu \norm{\vv^{R}}_{\dot{H}^{3/2}}^{2} &\leq \norm{(\BB^{R} \cdot \Grad) \BB^{R}}_{\dot{H}^{1/2}} \norm{\vv^{R}}_{\dot{H}^{1/2}} \\
&\leq \norm{\BB^{R}}_{\dot{B}^{3/2}_{2,1}}^{2} \norm{\vv^{R}}_{\dot{H}^{1/2}}
\end{align*}
by \eqref{eqn:FtoB}. Dropping the second term on the left-hand side yields
\[
\frac{\rd}{\rd t} \norm{\vv^{R}}_{\dot{H}^{1/2}} \leq \norm{\BB^{R}}_{\dot{B}^{3/2}_{2,1}}^{2}.
\]
For any $T > t > 0$, integrating in time over $[0,t]$ and taking the supremum over $t \in [0,T]$ yields
\[
\sup_{s \in [0,T]} \norm{\vv^{R}(s)}_{\dot{H}^{1/2}} \leq \int_{0}^{T} \norm{\BB^{R}(s)}_{\dot{B}^{3/2}_{2,1}}^{2} \, \rd s.
\]
This implies that
\[
\norm{\vv^{R}(t)}_{\dot{H}^{3/2}}^{2} \leq \frac{1}{\nu} \norm{\BB^{R}(t)}_{\dot{B}^{3/2}_{2,1}}^{2} \int_{0}^{T} \norm{\BB^{R}(s)}_{\dot{B}^{3/2}_{2,1}}^{2} \, \rd s,
\]
so that
\begin{equation}
\label{eqn:VEstimate}
\sup_{s \in [0,T]} \norm{\vv^{R}(s)}^{2}_{\dot{H}^{1/2}} + 2 \nu \int_{0}^{T} \norm{\vv^{R}(s)}_{\dot{H}^{3/2}}^{2} \, \rd s \leq 3 \left( \int_{0}^{T} \norm{\BB^{R}(s)}_{\dot{B}^{3/2}_{2,1}}^{2} \, \rd s \right)^{2}.
\end{equation}
Hence by interpolation,
\begin{equation}
\label{eqn:VEstimate2}
\int_{0}^{T} \norm{\vv^{R}(s)}_{\dot{H}^{1}}^{4} \, \rd s \leq \frac{1}{\nu} \left( \int_{0}^{T} \norm{\BB^{R}(s)}_{\dot{B}^{3/2}_{2,1}}^{2} \, \rd s \right)^{4}.
\end{equation}

Now, let $T \leq T_{1}$ be any time such that assumption~\eqref{eqn:Assumption} holds. Then we obtain
\begin{equation}
\label{eqn:VEstimateTime}
\int_{0}^{T} \norm{\vv^{R}(s)}_{\dot{H}^{1}}^{4} \, \rd s \leq \frac{\nu^{3}}{16 c_{3} c_{4}}.
\end{equation}
Combining \eqref{eqn:HEstimateTime} and \eqref{eqn:VEstimateTime} yields \eqref{eqn:DesiredBound} with $T_{0} = T$, and hence $T(R) \geq T$ for all such $T$; in particular, $T(R) \geq T_{1}$.

Moreover, \eqref{eqn:WEstimate2} holds on the interval $[0, T]$, and substituting \eqref{eqn:HEstimate2} and \eqref{eqn:VEstimate2} into \eqref{eqn:WEstimate2} yields
\begin{align}
&\sup_{s \in [0,T]} \norm{\ww^{R}(s)}_{\dot{H}^{1/2}}^{2} +  \nu \int_{0}^{T} \norm{\ww^{R}(s)}_{\dot{H}^{3/2}}^{2} \, \rd s \notag\\
&\qquad \leq \frac{8c_{3}}{\nu^{2}} \norm{\uu_{0}}_{\dot{H}^{1/2}}^{4} + \frac{4c_{3}}{\nu^{2}}  \left( \int_{0}^{T} \norm{\BB^{R}(s)}_{\dot{B}^{3/2}_{2,1}}^{2} \, \rd s \right)^{4}. \label{eqn:WEstimate3}
\end{align}
Hence, using \eqref{eqn:HEstimate}, \eqref{eqn:VEstimate} and \eqref{eqn:WEstimate3}, we obtain
\begin{align*}
&\sup_{s \in [0,T]} \norm{\uu^{R}(s)}_{\dot{H}^{1/2}}^{2} + 2 \nu \int_{0}^{T} \norm{\Grad \uu^{R}(s)}_{\dot{H}^{1/2}}^{2} \, \rd s \\
&\qquad \qquad \leq \sup_{s \in [0,T]} \norm{\hh^{R}(s)}_{\dot{H}^{1/2}}^{2} + 2 \nu \int_{0}^{T} \norm{\Grad \hh^{R}(s)}_{\dot{H}^{1/2}}^{2} \, \rd s  \\ 
&\qquad \qquad \qquad + \sup_{s \in [0,T]} \norm{\vv^{R}(s)}_{\dot{H}^{1/2}}^{2} + 2 \nu \int_{0}^{T} \norm{\Grad \vv^{R}(s)}_{\dot{H}^{1/2}}^{2} \, \rd s  \\ 
&\qquad \qquad \qquad + \sup_{s \in [0,T]} \norm{\ww^{R}(s)}_{\dot{H}^{1/2}}^{2} + 2 \nu \int_{0}^{T} \norm{\Grad \ww^{R}(s)}_{\dot{H}^{1/2}}^{2} \, \rd s \\
&\qquad \qquad \leq 2 \norm{\uu_{0}}_{\dot{H}^{1/2}}^{2} + \frac{16c_{3}}{\nu^{2}} \norm{\uu_{0}}_{\dot{H}^{1/2}}^{4} \\
&\qquad \qquad \qquad + 3 \left( \int_{0}^{T} \norm{\BB^{R}(s)}_{\dot{B}^{3/2}_{2,1}}^{2} \, \rd s \right)^{2} + \frac{8c_{3}}{\nu^{2}}  \left( \int_{0}^{T} \norm{\BB^{R}(s)}_{\dot{B}^{3/2}_{2,1}}^{2} \, \rd s \right)^{4}.
\end{align*}
This completes the proof.
\end{proof}

Proposition~\ref{prop:AuxEstimate} appears to show that the existence time for the $\uu$ equation depends on the existence time for the $\BB$ equation; but it is clear from Proposition~\ref{prop:B-Estimate} that the existence time for the $\BB$ equation ought to depend on the existence time for the $\uu$ equation. In order to circumvent this seemingly circular argument, we now show that there is some (short) time interval such that \eqref{eqn:Assumption} holds for all $R > 0$.

\begin{lemma}
\label{lem:AuxTime}
There is a time $T_{2} = T_{2}(\nu, \norm{\uu_{0}}_{\dot{B}^{1/2}_{2,1}}, \norm{\BB_{0}}_{\dot{B}^{3/2}_{2,1}}) > 0$ such that
\[
\int_{0}^{T} \norm{\BB^{R}(s)}_{\dot{B}^{3/2}_{2,1}}^{2} \, \rd s \leq \frac{\nu}{2 (c_{3} c_{4})^{1/4}} =: C_{*}
\]
for all $T \leq \min\{ T_{1}, T_{2} \}$ and all $R > 0$.
\end{lemma}

\begin{proof}
Define
\[
Z_{R}(t) := \int_{0}^{t} \norm{\BB^{R}(s)}_{\dot{B}^{3/2}_{2,1}}^{2} \, \rd s.
\]
Using the estimate on $\norm{\BB^{R}(s)}_{\dot{B}^{3/2}_{2,1}}$ from Proposition~\ref{prop:B-Estimate}, we obtain
\[
Z_{R}(t) \leq t \norm{\BB_{0}}_{\dot{B}^{3/2}_{2,1}}^{2} \exp \left( 2c_{1} \int_{0}^{t} \norm{\Grad \uu(s)}_{\dot{B}^{3/2}_{2,1}} \, \rd s \right).
\]
Using the estimate on $\int_{0}^{t} \norm{\Grad \uu(s)}_{\dot{B}^{3/2}_{2,1}} \, \rd s$ from Proposition~\ref{prop:U-Estimate}, we obtain
\begin{align}
Z_{R}(t) &\leq t \norm{\BB_{0}}_{\dot{B}^{3/2}_{2,1}}^{2} \exp \bigg( \frac{2c_{1}}{\nu} \norm{\uu_{0}}_{\dot{B}^{1/2}_{2,1}} + \frac{2c_{1}c_{2}}{\nu} \int_{0}^{t} \norm{\Grad \uu^{R}(s)}_{\dot{H}^{1/2}}^{2} \, \rd s \notag \\
&\qquad \qquad \qquad \qquad \qquad \qquad \qquad \qquad \qquad+ \frac{2c_{1}c_{2}}{\nu} Z_{R}(t) \bigg). \label{eqn:ZR-Estimate2}
\end{align}
Recall from \eqref{eqn:Assumption} that $C_{*} := \frac{\nu}{2 (c_{3} c_{4})^{1/4}}$.
Let
\begin{align*}
T_{2} &:= \frac{C_{*}}{\norm{\BB_{0}}_{\dot{B}^{3/2}_{2,1}}^{2}} \exp \bigg( - \frac{2c_{1}}{\nu} \norm{\uu_{0}}_{\dot{B}^{1/2}_{2,1}} - \frac{2c_{1}c_{2}}{\nu^{2}} \norm{\uu_{0}}_{\dot{B}^{1/2}_{2,1}}^{2} - \frac{8 c_{1}c_{2}c_{3}}{\nu^{4}} \norm{\uu_{0}}_{\dot{B}^{1/2}_{2,1}}^{4} \\
&\qquad \qquad \qquad \qquad \qquad - \frac{3c_{1}c_{2}}{\nu^{2}} C_{*} - \frac{2c_{1}c_{2}}{\nu} C_{*}^{2} - \frac{4c_{1}c_{2}c_{3}}{\nu^{4}} C_{*}^{4} \bigg).
\end{align*}
Suppose $t < \min\{ T_{1}, T_{2} \}$ and $Z_{R}(t) \leq C_{*}$. Then using Proposition~\ref{prop:AuxEstimate} to estimate the term $\int_{0}^{t} \norm{\Grad \uu^{R}(s)}_{\dot{H}^{1/2}}^{2} \, \rd s$, from \eqref{eqn:ZR-Estimate2} we obtain
\begin{align*}
Z_{R}(t) &\leq t \norm{\BB_{0}}_{\dot{B}^{3/2}_{2,1}}^{2} \exp \bigg( \frac{2c_{1}}{\nu} \norm{\uu_{0}}_{\dot{B}^{1/2}_{2,1}} + \frac{2c_{1}c_{2}}{\nu^{2}} \norm{\uu_{0}}_{\dot{B}^{1/2}_{2,1}}^{2} + \frac{16 c_{1}c_{2}c_{3}}{\nu^{4}} \norm{\uu_{0}}_{\dot{B}^{1/2}_{2,1}}^{4} \\
&\qquad \qquad \qquad \qquad \quad + \frac{3c_{1}c_{2}}{\nu^{2}} Z_{R}(t) + \frac{2c_{1}c_{2}}{\nu} [Z_{R}(t)]^{2} + \frac{8c_{1}c_{2}c_{3}}{\nu^{4}} [Z_{R}(t)]^{4} \bigg) \\
&\leq t \norm{\BB_{0}}_{\dot{B}^{3/2}_{2,1}}^{2} \exp \bigg( \frac{2c_{1}}{\nu} \norm{\uu_{0}}_{\dot{B}^{1/2}_{2,1}} + \frac{2c_{1}c_{2}}{\nu^{2}} \norm{\uu_{0}}_{\dot{B}^{1/2}_{2,1}}^{2} + \frac{16 c_{1}c_{2}c_{3}}{\nu^{4}} \norm{\uu_{0}}_{\dot{B}^{1/2}_{2,1}}^{4} \\
&\qquad \qquad \qquad \qquad \quad + \frac{3c_{1}c_{2}}{\nu^{2}} C_{*} + \frac{2c_{1}c_{2}}{\nu} C_{*}^{2} + \frac{8c_{1}c_{2}c_{3}}{\nu^{4}} C_{*}^{4} \bigg) \\
&< C_{*}.
\end{align*}
As $Z_{R}(t)$ is continuous and $Z_{R}(0) = 0$, this means that $Z_{R}(t)$ can never equal $C_{*}$ as long as $0 \leq t < \min \{ T_{1}, T_{2} \}$, and hence $Z_{R}(t) < C_{*}$ for all $0 \leq t < \min \{ T_{1}, T_{2} \}$.
\end{proof}

Combining the energy estimate~\eqref{eqn:Energy} with Proposition~\ref{prop:AuxEstimate} and Lemma~\ref{lem:AuxTime}, we obtain the following bound on $\int_{0}^{t} \norm{\uu^{R}(s)}_{H^{3/2}}^{2} \, \rd s$:
\begin{align}
&\int_{0}^{t} \norm{\uu^{R}(s)}_{H^{3/2}}^{2} \, \rd s \notag \\
&\qquad \leq \int_{0}^{t} \norm{\uu^{R}(s)}_{L^{2}}^{2} \, \rd s + \int_{0}^{t} \norm{\Grad \uu^{R}(s)}_{\dot{H}^{1/2}}^{2} \, \rd s \notag \\
&\qquad \leq 2t (\norm{\uu_{0}}_{L^{2}}^{2} + \norm{\BB_{0}}_{L^{2}}^{2}) + \frac{1}{\nu} \norm{\uu_{0}}_{\dot{H}^{1/2}}^{2} + \frac{8c_{3}}{\nu^{3}} \norm{\uu_{0}}_{\dot{H}^{1/2}}^{4} \notag \\
&\qquad \qquad + \frac{3}{2\nu} \left( \int_{0}^{t} \norm{\BB^{R}(s)}_{\dot{B}^{3/2}_{2,1}}^{2} \, \rd s \right)^{2} + \frac{4c_{3}}{\nu^{3}}  \left( \int_{0}^{t} \norm{\BB^{R}(s)}_{\dot{B}^{3/2}_{2,1}}^{2} \, \rd s \right)^{4} \label{eqn:Besov3DEnergy}
\end{align}
for all $0 \leq t \leq \min \{ T_{1}, T_{2} \}$.

We can now proceed analogously to the 2D case and show that $\uu^{R}$ and $\BB^{R}$ are uniformly bounded in the corresponding Besov spaces, although the algebra is slightly more involved.


\begin{thm}
\label{thm:UniformBesov3D}
Let $n=3$, and let $(\uu^{R}, \BB^{R})$ be the solution to \eqref{eqn:MHD-Cutoff-Besov}. There is a time $T_{*} = T_{*} (\nu, \uu_{0}, \norm{\BB_{0}}_{B^{3/2}_{2,1}}) > 0$ such that
\begin{align*}
\uu^{R} \text{ is uniformly bounded in } & L^{\infty}(0, T_{*}; \dot{B}^{1/2}_{2,1}(\R^{3})) \cap L^{1}(0, T_{*}; \dot{B}^{5/2}_{2,1}(\R^{3})), \\
\BB^{R} \text{ is uniformly bounded in } & L^{\infty}(0, T_{*}; \dot{B}^{3/2}_{2,1}(\R^{3})).
\end{align*}
\end{thm}

\begin{proof}
Let
\begin{align*}
M_{1} &= \norm{\uu_{0}}_{\dot{B}^{1/2}_{2,1}} + \frac{c_{2}}{\nu} \norm{\uu_{0}}_{\dot{B}^{1/2}_{2,1}}^{2} + \frac{8c_{2}c_{3}}{\nu^{3}} \norm{\uu_{0}}_{\dot{B}^{1/2}_{2,1}}^{4} \\
M_{2} &= 2c_{2} (\norm{\uu_{0}}_{L^{2}}^{2} + \norm{\BB_{0}}_{L^{2}}^{2}).
\end{align*}
Substituting from equation~\eqref{eqn:Besov3DEnergy} into Proposition~\ref{prop:U-Estimate}, when $t \leq \min \{ T_{1}, T_{2} \}$ we obtain
\begin{align*}
&\norm{\uu^{R}(t)}_{\dot{B}^{1/2}_{2,1}} + \nu \int_{0}^{t} \norm{\Grad \uu^{R}(s)}_{\dot{B}^{3/2}_{2,1}} \, \rd s \\
&\qquad \qquad \leq M_{1} + M_{2} t + c_{2} Z_{R}(t) + \frac{3c_{2}}{2\nu} (Z_{R}(t))^{2} + \frac{4c_{2}c_{3}}{\nu^{3}} (Z_{R}(t))^{4},
\end{align*}
where $Z_{R}(t) := \int_{0}^{t} \norm{\BB^{R}(s)}_{\dot{B}^{3/2}_{2,1}}^{2} \, \rd s$ as above. Letting $M_{3} = \norm{\BB_{0}}_{\dot{B}^{3/2}_{2,1}}^{2}$, Proposition~\ref{prop:B-Estimate} yields
\begin{align*}
Z_{R}(t) &\leq M_{3} \int_{0}^{t} \exp \left( 2c_{1} \int_{0}^{\tau} \norm{\Grad \uu^{R}(s)}_{\dot{B}^{3/2}_{2,1}} \, \rd s \right) \, \rd \tau \\
&\leq M_{3} t \exp \left( 2c_{1} \int_{0}^{t} \norm{\Grad \uu^{R}(s)}_{\dot{B}^{3/2}_{2,1}} \, \rd s \right).
\end{align*}
Setting
\begin{align*}
X_{R}(t) &= \norm{\uu^{R}(t)}_{\dot{B}^{1/2}_{2,1}}, \\
Y_{R}(t) &= \nu \int_{0}^{t} \norm{\Grad \uu^{R}(s)}_{\dot{B}^{3/2}_{2,1}} \, \rd s,
\end{align*}
yields
\begin{align}
X_{R}(t) + Y_{R}(t) &\leq M_{1} + M_{2} t + c_{2} M_{3} t \exp (2c_{1} Y_{R}(t) / \nu) \notag \\ 
&\qquad + \frac{3c_{2}}{2\nu} M_{3}^{2} t^{2}\exp (4c_{1} Y_{R}(t) / \nu) + \frac{4c_{2}c_{3}}{\nu^{3}} M_{3}^{4} t^{4} \exp (8c_{1} Y_{R}(t) / \nu). \label{eqn:ODE-XY-3D}
\end{align}

Let
\[
M_{4} = (2+c_{2}) M_{1} + \frac{3c_{2}}{2\nu} M_{1}^{2} + \frac{4c_{2}c_{3}}{\nu^{3}} M_{1}^{4},
\]
and set
\[
T_{*} = \min \left\{ T_{1}, T_{2}, \frac{M_{1}}{M_{2}}, \frac{M_{1}}{M_{3}} \exp (-2c_{1} M_{4}/\nu) \right\}.
\]
It suffices to show that $X_{R}(t) + Y_{R}(t) \leq M_{4}$ for all $t \in [0, T_{*})$ and all $R > 0$. To see this, note that $Y_{R}(t)$ is continuous and $Y_{R}(0) = 0$. Now, suppose $t < T_{*}$ and $Y_{R}(t) \leq M_{4}$; then
\begin{align*}
Y_{R}(t) &\leq M_{1} + M_{2}t + c_{2} M_{3} t \exp (2c_{1} Y_{R}(t) / \nu) \\ 
&\quad + \frac{3c_{2}}{2\nu} M_{3}^{2} t^{2}\exp (4c_{1} Y_{R}(t) / \nu) + \frac{4c_{2}c_{3}}{\nu^{3}} M_{3}^{4} t \exp (8c_{1} Y_{R}(t) / \nu) \\
&< M_{1} + M_{1} + c_{2} M_{1} \exp \left( \frac{2c_{1}}{\nu} [Y_{R}(t) - M_{4}] \right) \\
&\quad + \frac{3c_{2}}{2\nu} M_{1}^{2} \exp \left( \frac{4c_{1}}{\nu} [Y_{R}(t) - M_{4}] \right) + \frac{4c_{2}c_{3}}{\nu^{3}} M_{1}^{4} \exp \left( \frac{8c_{1}}{\nu} [Y_{R}(t) - M_{4}] \right) \\
&\leq M_{4}.
\end{align*}
This means that $Y_{R}(t)$ can never equal $M_{4}$ on the interval $[0, T_{*})$, hence $Y_{R}(t) < M_{4}$ for all $t \in [0, T_{*})$. The result follows from inequality \eqref{eqn:ODE-XY-3D} and Proposition~\ref{prop:B-Estimate}.
\end{proof}

Notice that, in the 3D case, $T_{1}$ (and hence $T_{*}$) depends on $\uu_{0}$ itself, and not just on the norm $\norm{\uu_{0}}_{B^{1/2}_{2,1}}$.


\section{Existence Proof}
\label{sec:Outline}

In summary, in either the 2D or the 3D case, there is some time $T_{*}$ such that
\begin{subequations}
\label{eqn:UniformBounds-Hom}
\begin{align}
\uu^{R} \text{ is uniformly bounded in } & L^{\infty}(0, T_{*}; \dot{B}^{n/2-1}_{2,1}(\R^{n})) \cap L^{1}(0, T_{*}; \dot{B}^{n/2+1}_{2,1}(\R^{n})), \\
\BB^{R} \text{ is uniformly bounded in } & L^{\infty}(0, T_{*}; \dot{B}^{n/2}_{2,1}(\R^{n})).
\end{align}
\end{subequations}
Having obtained these uniform bounds, in this section we outline the proof of Theorem~\ref{thm:MHDLocalExistence-Besov}, using broadly the same method as in Section~4.2 of \cite*{art:ARMA} to show the existence of a weak solution.

Let us first note that since the initial data is taken in inhomogeneous Besov spaces, the standard energy estimate~\eqref{eqn:Energy} implies that $\uu^{R}$ and $\BB^{R}$ are uniformly bounded in $L^{\infty}(0, T; L^{2}(\R^{n}))$ for any $T > 0$, and hence the uniform bounds \eqref{eqn:UniformBounds-Hom} imply that
\begin{subequations}
\label{eqn:UniformBounds-Inh}
\begin{align}
\uu^{R} \text{ is uniformly bounded in } & L^{\infty}(0, T_{*}; B^{n/2-1}_{2,1}(\R^{n})) \cap L^{1}(0, T_{*}; B^{n/2+1}_{2,1}(\R^{n})), \\
\BB^{R} \text{ is uniformly bounded in } & L^{\infty}(0, T_{*}; B^{n/2}_{2,1}(\R^{n})).
\end{align}
\end{subequations}

\subsection{Bounds on the Time Derivatives}

We first obtain uniform bounds on the time derivatives $\frac{\pd \uu^{R}}{\pd t}$ and $\frac{\pd \BB^{R}}{\pd t}$. By first applying the Leray projector $\Pi$ to the equations, we may eliminate the pressure term in \eqref{eqn:MHD-Cutoff-Besov} and consider the equations 
\begin{subequations}
\label{eqn:MHD-CutProj-Besov}
\begin{align}
\frac{\pd \uu^{R}}{\pd t} - \nu \Pi \Laplace \uu^{R} &= \Ss_{R}\Pi[(\BB^{R} \cdot \Grad) \BB^{R}] - \Ss_{R}\Pi[(\uu^{R} \cdot \Grad) \uu^{R}], \label{eqn:MHD-CutProj-Besov-u} \\
\frac{\pd \BB^{R}}{\pd t} &= \Ss_{R}\Pi[(\BB^{R} \cdot \Grad) \uu^{R}] - \Ss_{R}\Pi[(\uu^{R} \cdot \Grad) \BB^{R}]. \label{eqn:MHD-CutProj-Besov-B}
\end{align}
\end{subequations}

Taking the $\dot{B}^{n/2-1}_{2,1}$ norm of both sides of \eqref{eqn:MHD-CutProj-Besov-u} yields
\begin{align*}
\bignorm{\frac{\pd \uu^{R}}{\pd t}}_{\dot{B}^{n/2-1}_{2,1}} &\leq \nu \bignorm{\Laplace \uu^{R}}_{\dot{B}^{n/2-1}_{2,1}} \\
&\qquad \qquad + \bignorm{(\BB^{R} \cdot \Grad) \BB^{R}}_{\dot{B}^{n/2-1}_{2,1}} + \bignorm{(\uu^{R} \cdot \Grad) \uu^{R}}_{\dot{B}^{n/2-1}_{2,1}} \\
&\leq \nu \underbrace{\bignorm{\uu^{R}}_{\dot{B}^{n/2+1}_{2,1}}}_{\in L^{1}(0, T_{*})} + \underbrace{\bignorm{\BB^{R}}_{\dot{B}^{n/2}_{2,1}}^{2}}_{\in L^{\infty}(0, T_{*})} + \underbrace{\bignorm{\uu^{R}}_{\dot{B}^{n/2}_{2,1}}^{2}}_{\in L^{1}(0, T_{*})}
\end{align*}
where we have used the fact that, by interpolation, the uniform bounds \eqref{eqn:UniformBounds-Hom} imply that $\uu^{R}$ is uniformly bounded in $L^{2}(0, T_{*}; \dot{B}^{n/2}_{2,1}(\R^{n}))$. Similarly, taking the $\dot{B}^{n/2-1}_{2,1}$ norm of both sides of \eqref{eqn:MHD-CutProj-Besov-B} yields
\begin{align*}
\bignorm{\frac{\pd \BB^{R}}{\pd t}}_{\dot{B}^{n/2-1}_{2,1}} &\leq 2 \underbrace{\bignorm{\BB^{R}}_{\dot{B}^{n/2}_{2,1}}}_{\in L^{\infty}(0, T_{*})} \cdot \underbrace{\bignorm{\uu^{R}}_{\dot{B}^{n/2}_{2,1}}}_{\in L^{2}(0, T_{*})}.
\end{align*}
Hence
\begin{subequations}
\label{eqn:UniformBounds-TimeDeriv-Hom}
\begin{align}
\frac{\pd \uu^{R}}{\pd t} \text{ is uniformly bounded in } & L^{1}(0, T_{*}; \dot{B}^{n/2-1}_{2,1}(\R^{n})), \\
\frac{\pd \BB^{R}}{\pd t} \text{ is uniformly bounded in } & L^{2}(0, T_{*}; \dot{B}^{n/2-1}_{2,1}(\R^{n})).
\end{align}
\end{subequations}
Repeating these bounds using the inhomogeneous norms and \eqref{eqn:UniformBounds-Inh} implies that
\begin{subequations}
\label{eqn:UniformBounds-TimeDeriv-Inh}
\begin{align}
\frac{\pd \uu^{R}}{\pd t} \text{ is uniformly bounded in } & L^{1}(0, T_{*}; B^{n/2-1}_{2,1}(\R^{n})), \\
\frac{\pd \BB^{R}}{\pd t} \text{ is uniformly bounded in } & L^{2}(0, T_{*}; B^{n/2-1}_{2,1}(\R^{n})).
\end{align}
\end{subequations}

\subsection{Strong Convergence}

Using the uniform bounds \eqref{eqn:UniformBounds-Inh} and \eqref{eqn:UniformBounds-TimeDeriv-Inh}, one may use the Banach--Alaoglu theorem to extract a weakly-$*$ convergent subsequence such that
\begin{align*}
\uu^{R_{m}} &\weakstarto \uu && \text{ in } L^{\infty}(0, T_{*}; B^{n/2-1}_{2,1}(\R^{n})) \cap L^{1}(0, T_{*}; B^{n/2+1}_{2,1}(\R^{n})),\\
\BB^{R_{m}} &\weakstarto \BB && \text{ in } L^{\infty}(0, T_{*}; B^{n/2}_{2,1}(\R^{n})),\\
\frac{\pd \uu^{R_{m}}}{\pd t} &\weakstarto \frac{\pd \uu}{\pd t} && \text{ in } L^{1}(0, T_{*}; B^{n/2-1}_{2,1}(\R^{n})), \\
\frac{\pd \BB^{R_{m}}}{\pd t} &\weakstarto \frac{\pd \BB}{\pd t} && \text{ in } L^{2}(0, T_{*}; B^{n/2-1}_{2,1}(\R^{n})).
\end{align*}
We now show that $(\uu, \BB)$ is a \emph{weak} solution of the equations. By embedding the Besov spaces $B^{s}_{2,1}$ in the corresponding Sobolev spaces $H^{s}$, and using a variant of the Aubin--Lions compactness lemma (see Proposition~2.7 in \cite{book:CDGG}), there exists a subsequence of $(\uu^{R_{m}}, \BB^{R_{m}})$ that converges strongly in $L^{2}(0, T; H^{s}(K))$ for any $s \in (\frac{n}{2} - 1, \frac{n}{2})$ and any compact subset $K \subset \R^{n}$; and thus they also converge strongly in $L^{2}(0, T; L^{2}(K))$, and hence the limit satisfies
\[
\uu, \BB \in L^{\infty}(0, T; L^{2}(\R^{n})) \cap L^{2}(0, T; V(\R^{n})).
\]
This local strong convergence allows us to pass to the limit in the nonlinear terms: an argument similar to Proposition~4.5 in \cite{art:ARMA} will show that (after passing to a subsequence)
\[
\Ss_{R_{m}}[(\uu^{R_{m}} \cdot \Grad) \BB^{R_{m}}] \weakstarto (\uu \cdot \Grad) \BB
\]
(and so on) in $L^{2}(0, T; V^{*}(\R^{n}))$ (see \S2.2.4 of \cite{book:CDGG} for full details). Thus $(\uu, \BB)$ is indeed a weak solution of \eqref{eqn:MHD-NonRes}.

\section{Uniqueness}
\label{sec:Uniqueness-Besov}

We now prove a uniqueness result in 3D.

\begin{prop}
\label{prop:Uniqueness-Besov}
Let $(\uu_{j}, \BB_{j})$, $j = 1, 2$, be two solutions of \eqref{eqn:MHD-NonRes} with the same initial conditions $\uu_{j}(0) = \uu_{0}$, $\BB_{j}(0) = \BB_{0}$, such that 
\begin{align*}
\uu_{j} &\in L^{\infty}(0, T_{*}; B^{1/2}_{2,1}(\R^{3})) \cap L^{1}(0, T_{*}; B^{5/2}_{2,1}(\R^{3})), \\
\BB_{j} &\in L^{\infty}(0, T_{*}; B^{3/2}_{2,1}(\R^{3})).
\end{align*}
Then $(\uu_{1}, \BB_{1}) = (\uu_{2}, \BB_{2})$ as functions in $L^{\infty}(0, T; L^{2}(\R^{3}))$.
\end{prop}

\begin{proof}
Take the equations for $(\uu_{1}, \BB_{1})$ and $(\uu_{2}, \BB_{2})$ and subtract: writing $\ww = \uu_{1} - \uu_{2}$, $\zz = \BB_{1} - \BB_{2}$ and $q = p_{1} - p_{2}$, we obtain
\begin{subequations}
\begin{align}
\frac{\pd \ww}{\pd t} + (\uu_{1} \cdot \Grad) \ww + (\ww \cdot \Grad) \uu_{2} - \nu \Laplace \ww + \Grad q &= (\BB_{1} \cdot \Grad) \zz + (\zz \cdot \Grad) \BB_{2}, \label{eqn:MHDDiff-Besov-u} \\
\frac{\pd \zz}{\pd t} + (\uu_{1} \cdot \Grad) \zz + (\ww \cdot \Grad) \BB_{2} &= (\BB_{1} \cdot \Grad) \ww + (\zz \cdot \Grad) \uu_{2}. \label{eqn:MHDDiff-Besov-B} 
\end{align}
\end{subequations}
Taking the inner product of \eqref{eqn:MHDDiff-Besov-u} with $\ww$ and \eqref{eqn:MHDDiff-Besov-B} with $\zz$, and adding, yields
\begin{align*}
&\frac{1}{2} \frac{\rd}{\rd t} \left( \norm{\ww}_{L^{2}}^{2} + \norm{\zz}_{L^{2}}^{2} \right) + \nu \norm{\Grad \ww}_{L^{2}}^{2} \\
&\qquad = \inner{(\zz \cdot \Grad) \BB_{2}}{\ww} - \inner{(\ww \cdot \Grad) \uu_{2}}{\ww} + \inner{(\zz \cdot \Grad) \uu_{2}}{\zz} - \inner{(\ww \cdot \Grad) \BB_{2}}{\zz} \\
&\qquad \leq \norm{\zz}_{L^{2}} \norm{\Grad \ww}_{L^{2}} \norm{\BB_{2}}_{L^{\infty}} + \norm{\ww}_{L^{2}} \norm{\Grad \ww}_{L^{2}} \norm{\uu_{2}}_{L^{\infty}} \\
&\qquad \qquad + \norm{\zz}_{L^{2}}^{2} \norm{\Grad \uu_{2}}_{L^{\infty}} + \norm{\ww}_{L^{6}} \norm{\Grad \BB_{2}}_{L^{3}} \norm{\zz}_{L^{2}} \\
&\qquad \leq \left( \norm{\ww}_{L^{2}} + \norm{\zz}_{L^{2}}\right) \norm{\Grad \ww}_{L^{2}} \Big( \norm{\uu_{2}}_{\dot{B}^{3/2}_{2,1}} + \norm{\BB_{2}}_{\dot{B}^{3/2}_{2,1}} \Big) + \norm{\zz}_{L^{2}}^{2} \norm{\Grad \uu_{2}}_{\dot{B}^{3/2}_{2,1}},
\end{align*}
so by Young's inequality
\begin{align*}
&\frac{\rd}{\rd t} \left( \norm{\ww}_{L^{2}}^{2} + \norm{\zz}_{L^{2}}^{2} \right) + \nu \norm{\Grad \ww}_{L^{2}}^{2} \\
&\qquad \leq \frac{c}{\nu} \Big( \norm{\uu_{2}}_{\dot{B}^{3/2}_{2,1}}^{2} + \norm{\BB_{2}}_{\dot{B}^{3/2}_{2,1}}^{2} + \norm{\Grad \uu_{2}}_{\dot{B}^{3/2}_{2,1}} \Big) \left( \norm{\ww}_{L^{2}}^{2} + \norm{\zz}_{L^{2}}^{2} \right)
\end{align*}
and uniqueness follows by Gronwall's inequality.
\end{proof}

Note, however, that this argument does not apply in 2D. This is because the term $\inner{(\ww \cdot \Grad) \BB_{2}}{\zz}$ cannot be estimated in the same way: in 3D we used the inequality
\[
\abs{\inner{(\ww \cdot \Grad) \BB_{2}}{\zz}} \leq \norm{\ww}_{L^{6}} \norm{\Grad \BB_{2}}_{L^{3}} \norm{\zz}_{L^{2}} \leq \norm{\Grad \ww}_{L^{2}} \norm{\BB_{2}}_{\dot{B}^{3/2}_{2,1}} \norm{\zz}_{L^{2}},
\]
but in 2D the best we can do is
\[
\abs{\inner{(\ww \cdot \Grad) \BB_{2}}{\zz}} \leq \norm{\ww}_{L^{\infty}} \norm{\Grad \BB_{2}}_{L^{2}} \norm{\zz}_{L^{2}} \leq \norm{\ww}_{L^{\infty}} \norm{\BB_{2}}_{\dot{B}^{1}_{2,1}} \norm{\zz}_{L^{2}},
\]
since the embedding $H^{1} \hookrightarrow L^{\infty}$ fails to hold in 2D. While we could use the embedding $\dot{B}^{1}_{2,1} \hookrightarrow L^{\infty}$, that would not allow us to absorb the term into the $\norm{\Grad \ww}_{L^{2}}$ term on the left-hand side.

This leaves us in the odd situation where we can prove uniqueness in 3D, but not in 2D! More importantly, however, it shows that a proof along the lines of \cite{art:Commutators} would not necessarily work, since the uniqueness proof is just a simpler version of the proof that the truncated solutions $(\uu^{R}, \BB^{R})$ are Cauchy in $L^{\infty}(0, T; L^{2}(\R^{n}))$.

\section{Conclusion}
\label{sec:Conclusion}

With initial data $\uu_{0} \in B^{n/2-1}_{2,1}(\R^{n})$ and $\BB_{0} \in B^{n/2}_{2,1}(\R^{n})$ for $n=2,3$, we have proved the existence of a solution $(\uu, \BB)$ satisfying
\begin{align*}
\uu &\in L^{\infty}(0, T_{*}; B^{n/2-1}_{2,1}(\R^{n})) \cap L^{1}(0, T_{*}; B^{n/2+1}_{2,1}(\R^{n})), \\
\BB &\in L^{\infty}(0, T_{*}; B^{n/2}_{2,1}(\R^{n})).
\end{align*}
It is clear, however, that there is considerable scope for further work in a number of directions. Firstly, while the a priori estimates in Section~\ref{sec:APriori2D3D} depend only on the norms of the initial data in the corresponding \emph{homogeneous} Besov spaces, that is $\norm{\uu_{0}}_{\dot{B}^{n/2-1}_{2,1}}$ and $\norm{\BB_{0}}_{\dot{B}^{n/2}_{2,1}}$, in 3D the use of the commutator estimate Lemma~\ref{lem:SobolevNonlinear} (from \cite{art:Chemin1992}) forces the use of inhomogeneous spaces.

It is thus natural to ask whether all three norms on the right-hand side of \eqref{eqn:SobolevNonlinear} could be taken in homogeneous spaces: if such a generalisation could be proved, then in 3D the a priori estimates could be closed up while assuming only that $\uu_{0} \in \dot{B}^{1/2}_{2,1}$ and $\BB_{0} \in \dot{B}^{3/2}_{2,1}$ (though further work would be required to obtain a bona fide solution, as the method of Section~\ref{sec:Outline} would no longer apply).

A partial generalisation of Lemma~\ref{lem:SobolevNonlinear} is proved in \cite{art:LowerBound}: it is shown that
\[
\abs{\inner{\Lambda^{s}[(\uu \cdot \Grad)\uu]}{\Lambda^{s}\uu}} \leq c \norm{\uu}_{\dot{H}^{s_{1}}} \norm{\uu}_{\dot{H}^{s_{2}}} \norm{\uu}_{\dot{H}^{s}},
\]
provided that $s \geq 1$ and $s_{1}, s_{2} > 0$ such that
\[
1 \leq s_{1} < \tfrac{n}{2} + 1 \qquad \text{and} \qquad s_{1} + s_{2} = s + \tfrac{n}{2} + 1.
\]
Unfortunately the case we would want to apply requires $s = n/2-1$, which does not satisfy $s \geq 1$ in 2D or 3D.

Secondly, it remains to prove that the solution whose existence is asserted in Theorem~\ref{thm:MHDLocalExistence-Besov} is unique in 2D. While it might be possible to adapt the proofs of the a priori estimates (Propositions~\ref{prop:B-Estimate} and~\ref{prop:U-Estimate}) by working in the space $B^{0}_{2,1}$, the argument relies on certain cancellations which are no longer available when considering the difference of solutions, and initial investigations suggest that such an approach will likely not succeed.

An alternative approach would be to recast the equations in a Lagrangian formulation and consider the particle trajectories of the magnetic field $\BB$. The Lagrangian approach, most notably applied to the Euler equations by \cite{art:Yudovich1963}, has yielded significant results in Besov spaces for both the Euler equations (due to \cite{art:Chae2004}) and for MHD. In particular, in proving existence and uniqueness of solutions to fully ideal MHD in the Besov space $B^{1+n/p}_{p,1}(\R^{n})$, \cite{art:MiaoYuan2006} use the volume-preservation of the push-forward along particle trajectories of $\uu + \BB$ and $\uu - \BB$ to yield uniqueness; such a method could perhaps be adapted to the non-resistive case.

\addcontentsline{toc}{section}{Bibliography}

\bibliographystyle{agsm}
\bibliography{BesovMHDPaper}

\end{document}